\documentclass[a4paper,11pt]{article} 

\usepackage{amssymb}
\usepackage{amsfonts}
\usepackage{amsmath}
\usepackage{mathrsfs}
\usepackage{graphicx} 
\usepackage{amsbsy}
\usepackage{theorem}
\usepackage{color}
\usepackage{hyperref}
\usepackage{tikz}
\usepackage[normalem]{ulem}
\usepackage[latin1]{inputenc}

\newtheorem{theorem}{Theorem}[section]
\newtheorem{lemma}[theorem]{Lemma}

\newtheorem{definition}[theorem]{Definition}
\newtheorem{proposition}[theorem]{Proposition}
\newtheorem{remark}[theorem]{Remark}
\newtheorem{example}[theorem]{Example}

\def\thetheorem{\thesection.\arabic{theorem}}
\def\thesection{\arabic{section}}

\def\theequation {\thesection.\arabic{equation}}

\def\beq{\begin{equation}\displaystyle}
\def\eeq{\end{equation}}
\def\bel{\begin{equation} \displaystyle \begin{array}{l} }
\def\eel{\end{array} \end{equation} }
\def\bell{\begin{equation} \displaystyle \begin{array}{ll} }
\def\eell{\end{array} \end{equation} }

\def\bea{\begin{eqnarray}}
\def\eea{\end{eqnarray} }
\def\bean{\begin{eqnarray*}}
\def\eean{\end{eqnarray*} }

\newenvironment{proof}{\noindent{\bf Proof.~}}
{{\mbox{}\hfill {\small \fbox{}}\\}}

\catcode`@=11
\renewcommand\appendix{\bigskip {\noindent \Large \bf Appendix}
\setcounter{section}{0}%
\setcounter{subsection}{0}%
\setcounter{equation}{0}%
\setcounter{theorem}{0}%
\def\thetheorem{A.\arabic{theorem}}
\def\theequation {A.\arabic{equation}}}
\catcode`@=12

\newcommand{\dv}{\mathop{\rm div}\nolimits}

\def\NN{\mathbb{N}}

\def\RR{\mathbb{R}}

\def\ZZ{\mathbb{Z}}
\def\E{\mathbb{E}}
\def\EE{\mathbb{E}}
\def\PP{\mathbb{P}}

\def\ds{\displaystyle}

\def\bs{\bigskip}

\def\eps{\varepsilon}

\def\pa{\partial}

\def\calA{\mathcal{A}}

\def\calM{\mathcal{M}}

\def\calP{\mathcal{P}}

\def\calT{\mathcal{T}}
\def\calV{\mathcal{V}}

\def\smes{\mathcal{S}_\mathcal{M}}

\def\achapo{\widehat{a}}

\begin{document}

\title{Convergence order of upwind type schemes for transport equations with discontinuous coefficients}

\author{
Fran\c{c}ois Delarue\thanks{Laboratoire J.-A. Dieudonn\'e, 
UMR CNRS 7351, 
Univ. Nice, Parc Valrose, 06108 Nice Cedex 02, France. Email: \texttt{delarue@unice.fr}}, 
Fr\'ed\'eric Lagouti\`ere\thanks{Laboratoire de Mathématiques d'Orsay, Univ. Paris-Sud, CNRS, Université
Paris-Saclay, 91405 Orsay, France. Email: \texttt{frederic.lagoutiere@math.u-psud.fr}},
Nicolas Vauchelet\thanks{Sorbonne Universit\'e, UPMC Univ Paris 06, Laboratoire Jacques-Louis Lions 
UMR CNRS 7598, Inria, F75005 Paris, France, Email: \texttt{nicolas.vauchelet@upmc.fr}}
}

\maketitle

\begin{abstract}
An analysis of the error of the upwind scheme for transport equation 
with discontinuous coefficients is provided. 
We consider here a velocity field that is bounded and one-sided Lipschitz continuous.
In this framework, solutions are defined in the sense of measures along the lines of 
Poupaud and Rascle's work.
We study the convergence order of 
the upwind scheme in the Wasserstein distances. More precisely, we prove that in this setting
the convergence order is $1/2$. We also show the optimality of this result. In the appendix, we show that this 
result also applies to other "diffusive" "first order" schemes and to a forward semi-Lagrangian scheme. 
\end{abstract}

\bs



\bs

\section{Introduction}

This paper is devoted to the numerical analysis of an upwind scheme for the
linear transport equation in conservative form (continuity equation) with discontinuous coefficients. 
In space dimension $d$, this equation reads 
\beq\label{eq:transp}
\pa_t\rho + \dv\big( a \rho\big)=0 , \qquad t>0,\quad x\in\RR^d,
\eeq
and is complemented with the initial condition $\rho(0,\cdot)=\rho^{ini}$.

We consider a rather weak regularity of the velocity, bounded and one-sided (right) Lipschitz continuous (OSL for short): 
$$
a \in L^\infty([0,+\infty);L^\infty(\RR^d))^d
$$
and there exists $\alpha \in L^1_{loc}([0,+\infty))$ such that 
\beq\label{eq:OSL}
\langle a(t,x)-a(t,y),x-y \rangle \leq \alpha(t) | x-y |^2, 
\eeq
where $\langle \cdot,\cdot\rangle$ stands for the Euclidean scalar product in $\RR^d$. 
Note that, as $a$ is assumed to be bounded, the right Lipschitz continuity coefficient $\alpha (t)$ 
is non-negative, for any $t > 0$. 

Since the velocity $a$ is not assumed Lipschitz continuous, the definition of solutions with the 
characteristic curves is not straightforward. 
In \cite{Filippov}, Filippov proposed
a notion of solution which extends the classical one. 
Using this so-called Filippov flow, Poupaud and Rascle proposed in \cite{PoupaudRascle} a notion
of solutions to the conservative linear transport equation \eqref{eq:transp}. They are defined as 
$Z_\# \rho^{ini}$ where $Z$ is the Filippov flow corresponding to the velocity $a$ (note that a stability 
result of the flow has been proved recently in \cite{Bianchini}).
In dimension $1$, these solutions are equivalent to
the duality solutions defined in \cite{bj1,BJpg}.
Duality solutions have also been defined in higher dimensions in \cite{bujama} 
but the theory is still not complete for the conservative transport equation.

In the present setting, solutions to the continuity equation can form Dirac masses and then are defined 
in the sense of measures. 
The numerical approximation of measure valued solutions to \eqref{eq:transp} requires a particular
care. In dimension $1$, Gosse \& James, \cite{GJ}, proposed a class of finite difference numerical schemes 
that includes the one dimensional upwind scheme.
Using the setting of duality solutions, the convergence of these schemes has been obtained in the sense 
of measures. However no error estimates are provided. 
More recently, Bouchut, Eymard \& Prignet, in \cite{BEP}, have proposed a different strategy in any dimension
with a finite volume scheme defined by the characteristics (the flow is assumed to be given).
The convergence is proved, on general admissible meshes, in the sense of measures, 
but no error estimates are provided. 

We here present an error analysis of an upwind scheme
for the continuity equation \eqref{eq:transp} when the coefficient $a$ is one-sided Lipschitz continuous.
More precisely, we prove that the order of convergence of the scheme is $1/2$ in Wasserstein distances $W_p$.

The convergence order of the upwind scheme 
for transport equations has received a lot of attention.
When the velocity field is Lipschitz continuous, this scheme is known to be first order 
convergent in the $L^\infty$ norm for any smooth initial data in ${\mathcal C}^2(\RR^d)$ {\em and for well-suited meshes},
provided a stability (Courant-Friedrichs-Lewy) condition holds: see \cite{bouche}.
However, for non-smooth initial data or on more general meshes, this order of convergence falls down to $1/2$, 
in $L^p$ norms.
This result has been first proved in the Cartesian framework by Kuznetsov in \cite{Kuznetsov} (this analysis is 
actually done for the entropy solutions of scalar (nonlinear) hyperbolic equations).
On quite general meshes, for $L^1(\RR^d)\cap BV(\RR^d)$ initial data, the result was tackled in \cite{MV} 
and \cite{caniveau}, in quite similar settings but with very different methods, in the $L^\infty$ in time 
and $L^1$ in space norm (we here will use the formalism developed in \cite{caniveau}). At last, let us mention that 
for initial data that are Lipschitz-continuous, the convergence to the $1/2 - \eps$ order in the $L^\infty$ norm 
(for any $\eps > 0$) is proved in \cite{M} and again in \cite{caniveau}. We also can mention \cite{Despres} and 
\cite{VV} for related results. 
We emphasize that the techniques used in \cite{M,MV} and \cite{caniveau} are totally different.
In the former, the technique is based on entropy estimates, whereas in the latter, the proof
relies on the construction of a stochastic characteristic defined as a Markov chain.

Up to our knowledge, there are no error estimates for the upwind scheme or, 
more generally, finite volume schemes, when the velocity field is less smooth. 
When the velocity field is given in $L^1((0,T);(W^{1,1}(\Omega))^d)$ for $\Omega \subset \RR^d$,
the {\em convergence} of numerical solutions, obtained thanks to an upwind scheme, towards renormalized 
solutions of the transport equation is studied in \cite{Boyer}. 

In this work we perform a numerical analysis of the upwind scheme in the weak framework where the 
velocity field is one-sided Lipschitz continuous.
More precisely, our main result is the following. 
\vspace{4pt}
\\
{\bf Result} (see Theorem \ref{TH} for a precise statement). 
Let $\rho^{ini}$ be a probability measure on $\RR^d$ such that $\int_{\RR^d} |x|^p \rho^{ini}(dx)<\infty$ 
for some $p\geq 1$.
We assume that the velocity field $a$ belongs to 
$L^\infty([0,+\infty);L^\infty(\RR^d))^d$ and satisfies the 
OSL condition \eqref{eq:OSL}. 
Let $\rho$ be the solution of the transport equation \eqref{eq:transp} with initial datum $\rho^{ini}$ 
(whose existence and uniqueness
is recalled in Theorem \ref{thPR} below). Let $\rho_{\Delta x}$ be the numerical approximation computed
thanks to the upwind scheme on a Cartesian grid mesh with space step $\Delta x$. Then, under a (usual)
Courant-Friedrichs-Lewy condition linking the time step $\Delta t$ and the cell size $\Delta x$, 
there exists a 
constant $C\geq 0$, depending only on 
$a$, $p$ and $\rho^{ini}$, such that we have 
\begin{equation}
\label{eq:final:estimate}
\forall\, t \geq 0, \qquad W_p\bigl(\rho(t),\rho_{\Delta x}(t)\bigr) \leq C \Bigl( \sqrt{t \Delta x} 
+ \Delta x \Bigr) \exp \biggl( 2 \int_{0}^t \alpha(s) ds \biggr).
\end{equation}

In this result, $W_p$ denotes the Wasserstein distance of order $p \geq 1$, whose definition will be
recalled below. 

\begin{remark}
A natural question is then: equipped with this result, and now assuming that the 
solution is, let us say, with 
bounded variation in space, do we recover a convergence order in a strong norm such as $L^1$ in space? 
The answer is yes, thanks to an interpolation estimate by Santambrogio (\cite{Filippo}), which proof is 
reported in 
the Appendix: let $f$, $g$ be two non-negative functions in $L^1(\RR^d)$ with mass equal to 1. 
There exists a constant $C \in \RR$ such that 
\[
|| f - g ||_{L^1} \leq C | f - g |_{BV}^{1/2} W_1(f,g)^{1/2}
\]
where $| \cdot |_{BV}$ denotes the total variation semi-norm on $\RR^d$. 
Thus, we recover a $1/4$ convergence order in $L^1$. This is not optimal: it is known that the 
convergence order is $1/2$, but, to reach $1/2$, one should use the additional smoothness of the solution in 
the $W_1$ estimate to obtain the convergence to the first order. 
\end{remark}

The main idea of the proof of the theorem is, in the spirit of \cite{caniveau},
to show that, similar to the exact solution, the numerical solution can be 
interpreted as the pushforward of the initial condition by a (numerical) flow. However, this flow is stochastic 
whilst that associated with the original equation is obviously deterministic. 
The numerical (deterministic) solution is
then represented as the {\em expectation} of the pushforward 
of the initial datum by the stochastic flow.

Finally, we emphasize that although our result is fully established for an upwind scheme,
the approach developed in this work can be easily extended to other schemes, as it is
explained in the appendix.
Meanwhile, it is worth mentioning that, 
although our strategy is shown to work on any general mesh for the semi-Lagrangian scheme discussed in the appendix, 
it works on a 
 Cartesian grid only for the upwind scheme under study.  
This is a major difference with \cite{caniveau}, in which the 
analysis of the upwind scheme is performed on a general mesh. The rationale for this difference is as follows:
In \cite{caniveau}, the strategy for handling 
the upwind scheme on non-Cartesian grids relies on a time reversal argument
and, somehow, on the analysis of the characteristics
associated with the velocity field $-a$. Whilst there is no difficulty for doing so in the Lipschitz setting, this is of course much more challenging under the weaker OSL condition 
\eqref{eq:OSL} since the ordinary differential equation 
driven by $-a$ is no more well-posed. We hope to address this question in future works.

The outline of the paper is the following. 
Section \ref{measure} is devoted to general definitions and notations 
that will be used throughout the
paper (in particular, we recall the notion of measure solutions to the transport equation 
\eqref{eq:transp} as defined in \cite{PoupaudRascle}). 
In Section \ref{sec:scheme}, we define the upwind scheme on a Cartesian mesh
and provide some basic properties for this scheme.
Section \ref{sec:ordre} is devoted to the statement and the proof of our main result: the convergence with
order $1/2$ of the upwind scheme on a Cartesian grid.
Finally, in order to illustrate the optimality of this order of convergence, we provide
in Section \ref{sec:num} first an explicit computation of the error in a simple case, and then 
some numerical experiments in dimension $1$.
An appendix provides an extension to other numerical (similar) schemes and the proof 
of the lemma used in the preceding remark.

\section{Measure solutions to the continuity equation}
\label{measure}

All along the paper, we will make use of the following notations.
We denote by 
$\calM_{b}(\RR^d)$ the space of finite signed measures on $\RR^d$ equipped with the 
Borel $\sigma$-field ${\mathcal B}(\RR^d)$. 
For $\rho \in {\cal M}_{b}(\RR^d)$, we denote by $|\rho|(\RR^d)$ its total variation, or total mass. 
The space of measures $\calM_b(\RR^d)$ is endowed with the weak topology $\sigma({\cal M}_b(\RR^d),{\mathcal C}_0(\RR^d))$, where 
${\mathcal C}_0(\RR^d)$ is the set of continuous function on $\RR^d$ that tend to 0 at $\infty$.
We then define $\smes :={\mathcal C}([0,+\infty);{\cal M}_b(\RR^d) 
- \sigma({\cal M}_b,{\mathcal C}_{0}(\RR^d)))$ and we equip it with the
topology of  
uniform convergence on finite intervals of the form $[0,T]$, 
with $T>0$. 

For $\rho$ a measure in $\calM_b(\RR^d)$ and $Z$ a measurable map 
(throughout the paper, {\em measurability}
is understood 
as measurability with respect to the Borel $\sigma$-fields when
these latter are not specified), 
we denote by $Z_\#\rho$ the pushforward
measure of $\rho$ by $Z$; by definition, it satisfies
$$
\int_{\RR^d} \phi(x)\, Z_\#\rho(dx) = \int_{\RR^d} \phi(Z(x))\,\rho(dx) 
\quad \mbox{for any } \phi \in {\mathcal C}_{0}(\RR^d), 
$$
or, equivalently, 
$$
Z_\#\rho(A) = \rho\bigl(Z^{-1}(A)\bigr) \quad \mbox{for any } A \in {\mathcal B}(\RR^d). 
$$

Moreover, we denote by $\calP(\RR^d)$ the subspace of
${\mathcal M}_{b}(\RR^d)$ 
made of 
 probability measures on $(\RR^d,{\mathcal B}(\RR^d))$. Also, we let $\calP_{p}(\RR^d)$ 
 the space of probability measures with finite
$p$-th order moment, $p\geq 1$:
$$
\calP_p(\RR^d) = \left\{\mu \in \calP(\RR^d) :  \int_{\RR^d} |x|^p \mu(dx) <\infty\right\}.
$$
Finally, for any probability space 
$(\Omega,\calA,\PP)$
and any integrable random variable $X$ from $(\Omega,\calA)$ into $(\RR,{\mathcal B}(\RR))$, 
 we denote by $\E(X)$ the expectation of $X$.
 
\subsection{Wasserstein distance}

The space $\calP_p(\RR^d)$ is endowed with the Wasserstein distance $W_p$ defined 
by (see e.g. \cite{Ambrosio,Villani1, Villani2})
\beq\label{defWp}
W_p(\mu,\nu)= \inf_{\gamma\in \Gamma(\mu,\nu)} 
\left\{\int_{\RR^d\times\RR^d} |y-x|^p\,\gamma(dx,dy)\right\}^{1/p}
 \eeq
where $\Gamma(\mu,\nu)$ is the set of measures on $\RR^d\times\RR^d$ 
with marginals $\mu$ and $\nu$, i.e.
\begin{align*}
\Gamma(\mu,\nu) = \left\{ \gamma\in \calP_p(\RR^d\times\RR^d); \ \forall\, 
\xi\in {\mathcal C}_0(\RR^d), \right. &
\int_{\RR^d\times\RR^d} \xi(y_1)\gamma(dy_1,dy_2) = \int_{\RR^d} \xi(y_1) \mu(dy_1), \\
& \left.\int_{\RR^d\times\RR^d} \xi(y_2)\gamma(dy_1,dy_2) = \int_{\RR^d} \xi(y_2) \nu(dy_2) \right\}.
\end{align*}
It is known that in the definition of $W_p$
the infimum is actually a minimum (see \cite{Villani1}). 
A map that fulfils the minimum in the definition \eqref{defWp}
of $W_p$ is called an {\it optimal plan}. The set of optimal plans is denoted by $\Gamma_0(\mu,\nu)$.
Thus for all $\gamma_0\in \Gamma_0(\mu,\nu)$, we have
$$
W_p(\mu,\nu)^p= \int_{\RR^d\times\RR^d} |y-x|^p\,\gamma_0(dx,dy).
$$

We will make use of the following properties of the Wasserstein distance.
Given two measurable maps $X,Y:\RR^d\to \RR^d$, we have the inequality
\begin{equation}
\label{eqWasser}
W_p(X_\#\mu,Y_\#\mu) \leq \|X-Y\|_{L^p(\mu)}.
\end{equation}
Indeed, 
$\pi=(X,Y)_\# \mu\in \Gamma(X_\#\mu,Y_\#\mu)$ 
and $\int_{\RR^d} |x-y|^p\,\pi(dx,dy)=\|X-Y\|_{L^p(\mu)}^p$.

More generally, for any probability space 
$(\Omega,{\mathcal A},\PP)$ 
and any two random variables 
$X,Y : \Omega \rightarrow \RR^d$ such that
$\EE[ \vert X \vert^p]$ and 
$\EE[ \vert Y \vert^p]$ are finite, we have 
\begin{equation}
\label{eqWasser:omega}
\bigl( W_p(X_\#\PP,Y_\#\PP) \bigr)^p \leq \EE \bigl( |X-Y|^p \bigr),
\end{equation}
where $X_\#\PP$ and $Y_\#\PP$ denote the respective 
distributions of $X$ and $Y$ under $\PP$. The proof 
follows from the same argument as above: $\pi = (X,Y)_{\#} \PP
\in \Gamma(X_\#\PP,Y_\#\PP)$ and 
$\int_{\RR^d} |x-y|^p\,\pi(dx,dy)=\EE[ \vert X-Y|^p]$.

\subsection{Weak measure solutions for linear conservation laws}

We recall in this section some useful results
on weak measure solutions to the conservative transport equation \eqref{eq:transp}, 
when driven by
an initial datum $\rho(0,\cdot) = \rho^{ini} \in {\mathcal M}_{b}(\RR^d)$
and a vector field $a$ that satisfies the OSL condition.

We start by the following definition of characteristics \cite{Filippov}:
\begin{definition}\label{DefFlow}
Let us assume that $a :  [0,+\infty) \times \RR^d \ni (t,x) \mapsto a(t,x)\in \RR^d$ is a (measurable) vector field. 
A Filippov characteristic $Z(\cdot;s,x)$ 
stemmed from $x\in \RR^d$ at time $s \geq 0$
is a continuous function $[s,+\infty) \ni t \mapsto Z(t;s,x) \in \RR^d$ such that
$Z(s;s,x)=x$, 
$\frac{\pa}{\pa t} Z(t;s,x)$ exists for a.e. $t\geq s$ and
$$
\frac{\pa}{\pa t} Z(t;s,x) \in \big\{ \textrm{\rm Convess} (a)(t,\cdot)\big\}
\bigl(Z(t;s,x)\bigr) \quad \mbox{for a.e. } t\geq s.
$$
From now on, we will use the notation $Z(t,x)=Z(t;0,x)$.
\end{definition}

In this definition, $\textrm{\rm Convess}(E)$ denotes the essential convex hull of
the set $E$: let us remind briefly the definition for the sake of completeness (see 
\cite{Filippov,AubinCellina} for more details). We denote by $\textrm{\rm Conv}(E)$ the classical convex
hull of $E$, i.e., the smallest closed convex set containing $E$. Given the vector field 
$a(t,\cdot):\RR^d\rightarrow\RR^d$, its essential convex hull at point $x$ is defined as
$$
\bigl\{ \textrm{\rm Convess} (a)(t,\cdot) \bigr\}(x)=\bigcap_{r>0} \bigcap_{N\in \mathcal{N}_0} 
\textrm{\rm Conv}\bigl[ a\bigl( t, B(x,r)\setminus N\bigr)\bigr]\,,
$$
where $\mathcal{N}_0$ is the set of zero Lebesgue measure sets.
Then, we have the following existence and uniqueness result of Filippov
characteristics under the assumption that the vector field $a$ is one-sided Lipschitz continuous.

\begin{theorem}  [\cite{Filippov}]
Let $a\in L_{loc}^1([0,+\infty);L^\infty(\RR^d))^d$ be a vector field satisfying the OSL condition \eqref{eq:OSL}.
Then there exists a unique Filippov flow $Z$
associated with this vector field, 
meaning that there exists a unique 
characteristic for any initial condition $(s,x) \in [0,+\infty) \times \RR^d$.
This flow does not depend on the choice of the representative 
(up to a $dt \otimes dx$ null set)
of the velocity field $a$ as long as this version 
satisfies the OSL condition pointwise.
Moreover, we have the semi-group property: For any $t,\tau,s \in [0,+\infty)$ such that $t \geq \tau \geq s$ and 
$x\in \RR^d$,
$$
Z(t;s,x) = Z(\tau;s,x) + \int_\tau^t a(\sigma,Z(\sigma;s,x))\,d\sigma.
$$
\end{theorem}

Importantly, we have the following Lipschitz 
continuous estimate on the Filippov characteristic:

\begin{lemma}[\cite{Filippov}]\label{lem:JacobZ}
Let $a \in 
L^1([0,+\infty),L^\infty(\RR^d))^d$
satisfy the OSL 
condition
\eqref{eq:OSL}
and 
$Z$ be the associated flow.
Then, for all $t \geq s$ in $[0,+\infty)$, we have
\beq\label{bound1Z}
L_Z(t,s) := \sup_{x, y\in \RR^d, x \neq y} \frac{ | Z(t,s,x) - Z(t,s,y) |}{| x - y | }  
\leq e^{\int_{s}^{t} \alpha(\sigma)\,d\sigma}
\eeq
($L_Z(t,s)$ is the Lipschitz constant of the flow $Z$). 
\end{lemma}

\begin{proof}
For $x, y \in \RR^d$, we compute
$$
\frac {d}{dt} | Z(t,s,x) -Z(t,s,y)|^2 = 2\bigl\langle a(t,Z(t,s,x))-a(t,Z(t,s,y)),
Z(t,s,x)-Z(t,s,y)\bigr\rangle.
$$
Using the OSL estimate \eqref{eq:OSL}, we deduce
$$
\frac {d}{dt} | Z(t,s,x) -Z(t,s,y)|^2 \leq 2 \alpha(s) | Z(t,s,x) -Z(t,s,y)|^2.
$$
Thanks to a Gr\"onwall lemma, we get
$$
| Z(t,s,x) -Z(t,s,y)|^2 \leq e^{\int_{s}^{t} 2 \alpha(\sigma)\,d\sigma} |x-y|^2,$$
which completes the proof.
\end{proof}

An important consequence of this result is the existence and uniqueness
of weak measure solutions for the conservative linear transport equation.
This has been obtained by Poupaud \& Rascle in \cite{PoupaudRascle}.
\begin{theorem}[\cite{PoupaudRascle}]\label{thPR}
Let $a\in L^1_{loc}([0,+\infty),L^\infty(\RR^d))^d$ be a vector field satisfying
the OSL condition \eqref{eq:OSL}. Then, for any $\rho^{ini}\in \calM_b(\RR^d)$,
there exists a unique measure solution $\rho$ in $\smes$ to the conservative
transport equation \eqref{eq:transp} with initial datum $\rho(0,\cdot) = \rho^{ini}$ 
such that $\rho(t)=Z(t)_\#\rho^{ini}$, where $Z$ is
the unique Filippov flow, i.e. for any $\phi\in {\mathcal C}_0(\RR^d)$, we have
$$
\int_{\RR^d} \phi(x) \rho(t,dx) = \int_{\RR^d} \phi(Z(t,x)) \rho^{ini}(dx),
\qquad \mbox{for} \quad t\geq 0.
$$
\end{theorem}
Note that actually, in this result, the expression $\rho(t)=Z(t)_\#\rho^{ini}$ is somehow 
understood as a {\em definition} of solution to the Cauchy problem.
From now on, 
we will interpret the solutions to 
\eqref{eq:transp}
in this sense. 

To conclude this section, we recall the stability estimate of the flow due to Bianchini and 
Gloyer \cite[Theorem 1.1]{Bianchini}. 
This estimate reads as a bound for the difference between 
the flows $Z_{1}$ and $Z_{2}$
associated with two velocity fields 
$a_1$ and $a_2$ in $L^1([0,\infty),L^\infty(\RR^d))^d$ 
satisfying the OSL condition \eqref{eq:OSL}.
For any $r>0$ any $x\in B(0,r)$, 
it holds that
\begin{equation}
\label{eq:bianchini:gloyer}
|Z_1(s,t,x)-Z_2(s,t,x)|^2 \leq C\int_s^t \|a_1(\sigma,.)-a_2(\sigma,.)\|^{1/d}_{L^1(B(0,2R))}\,d\sigma,
\end{equation}
where $R=r+a_\infty T$,
and $a_\infty=\max\{\|a_1\|_\infty,\|a_2\|_\infty\}$
and $C$ is a constant that only depends on the dimension.

Remark that this estimate, which is also proved in the same paper to be optimal, is a bad hint 
to obtain a result as the one we will prove here, because the stability of the characteristics with respect to 
perturbations of the velocity field decreases as the space dimension increases. 
Based on this estimate, 
one could imagine 
that a similar phenomenon 
should occur when estimating the error of a numerical scheme
for 
\eqref{eq:transp}. Indeed 
(as it will be the case in the next section), 
the analysis of the scheme should consist in
regarding the numerical solution as the solution of 
an equation of the same type as 
\eqref{eq:transp}
but driven by an approximating velocity field. 
Then, it would be tempting to compare both solutions by 
means of 
\eqref{eq:bianchini:gloyer}. However, our result shows that 
this strategy is non-optimal, at least for the scheme studied in the paper. 
Our analysis exploits the fact that, in our case,
the 
structure of the approximating velocity
is actually very close to that of the original velocity field. 

\section{Definition of the scheme and basic properties}
\label{sec:scheme}

\subsection{Numerical discretization}

From now on, we consider a velocity field $a\in L^\infty([0,+\infty),L^\infty(\RR^d))^d$ 
and we choose a representative $\achapo$ in the equivalence class of $a$ in 
$L^\infty([0,+\infty),L^\infty(\RR^d))$: $\achapo$ 
is 
defined everywhere
and is 
jointly measurable in time and space; it is $dt \otimes dx$ a.e. equal to $a$; and, it satisfies the condition 
\eqref{eq:OSL} everywhere.
In order to simplify the presentation, we will keep the notation $a$ instead of $\achapo$ and 
we write $a=(a_1,\ldots,a_d)$.

We denote by $\Delta t > 0$ the time step and consider a Cartesian grid with
step $\Delta x_i > 0$ in the $i$th direction, $i=1,\ldots,d$, and $\Delta x = \max_i \Delta x_i$. 
For $i = 1,\dots,d$, we note $e_i$ the $i$th vector of the canonical basis of $\RR^d$. 
We define the multi-indices $J=(J_1, \ldots, J_d)\in \ZZ^d$, the space cells $C_J = 
[(J_1-\frac{1}{2})\Delta x_1,(J_1+\frac{1}{2})\Delta x_1)\times \ldots 
[(J_d-\frac{1}{2})\Delta x_d,(J_d+\frac{1}{2})\Delta x_d)$ and their center 
$x_J = (J_1\Delta x_1, \ldots, J_d \Delta x_d)$. Finally, we set $t^n = n \Delta t$. 

For a given non-negative measure $\rho^{ini}\in
\calP(\RR^d)$, we define for $J\in \ZZ^d$, 
\beq\label{disrho0}
\rho_{J}^0 = \int_{C_J} \rho^{ini}(dx)\geq 0, 
\eeq 
which actually is to be understood as $\rho_{J}^0 = \rho^{ini} (C_J)$. 
Since $\rho^{ini}$ is a probability
measure, the total mass of the system is $\sum_{J\in \ZZ^d} \rho_{J}^0 = 1$. 
We denote by $\rho_J^n$ an approximation of the value $\rho(t^n) (C_J)$, for $J\in \ZZ^d$, 

and we propose to compute this approximation by using the upwind scheme, 
that is to say, we let by induction:
\begin{multline}\label{dis_num}
\ds \rho_{J}^{n+1} = \rho_{J}^n - \sum_{i=1}^{d} \frac{\Delta t}{\Delta x_i}
\big(({a_i}^n_{J})^+ \rho_{J}^n - ({a_i}^n_{J+e_i})^- \rho_{J+e_i}^n 
-({a_i}^n_{J-e_i})^+ \rho_{J-e_i}^n + ({a_i}^n_{J})^- \rho_{J}^n \big), \\
n \in \NN, J \in \ZZ^d. 
\end{multline}
The notation $(a)^+ = \max\{0,a\}$ stands for the positive part of the real number $a$
and $(a)^- = \max\{0,-a\}$ for the negative part.
The numerical velocity is defined, for $i=1,\ldots,d$, by
\begin{equation}
\label{def:ai}
{a_i}_{J}^n = \frac{1}{\Delta t} \int_{t^n}^{t^{n+1}} a_i(s,x_{J})\,ds.
\end{equation}

\begin{remark}\label{otherupwind}
There is at least one other traditional upwind scheme, 
for which the velocity is to be computed at the interface: $a_{i+1/2}^n$. 
The difficulty with this one is explained at the end of Section \ref{ofvs} in the appendix.
\end{remark}

\begin{remark}
The discretization of the velocity requires the computation of the mean value
in time of the velocity field in formula \eqref{def:ai}.
However, 
if one assumes the velocity field to be Lipschitz 
continuous in time, uniformly in space, then 
${a_i}_{J}^n$ can be replaced 
by $a_i(t^n,x_{J})$.  
We refer to Remark 
\ref{rem:a:Lip:3} below for a short account on the new form 
of the main estimate
\eqref{eq:final:estimate}.
\end{remark}

\begin{remark}\label{ex1D}
In dimension 1, Scheme \eqref{dis_num} reads
$$
\rho_{j}^{n+1} = \rho_j^n - \frac{\Delta t}{\Delta x}\Big((a_j^n)^+ \rho_j^n - (a_{j+1}^n)^- \rho_{j+1}^n 
- (a_{j-1}^n)^+ \rho_{j-1}^n + (a_j^n)^-\rho_j^n\Big).
$$
We will make use of the following interpretation of this scheme. 
Defining $\rho^n_{\Delta x} = \sum_{j\in \ZZ} \rho_j^n \delta_{x_j}$,
we construct the approximation at time $t^{n+1}$ with the following two steps:
\begin{itemize}
\item The Dirac mass $\rho_j^n$, located at position $x_j$, moves with velocity $a_j^n$ to the position 
$x_j+a_j^n \Delta t$. 
Assuming a Courant-Friedrichs-Lewy condition $||a||_\infty \Delta t \leq \Delta x$, the point $x_j+a_j^n\Delta t$
belongs to the interval $[x_j,x_{j+1}]$ if $a_j^n\geq 0$, or to the interval $[x_{j-1},x_{j}]$ if $a_j^n \leq 0$.
\item Then we split the mass $\rho_j^n$ between $x_j$ and $x_{j+1}$ if $a_j^n\geq 0$ 
or between $x_{j-1}$ and $x_j$ if $a_j^n\leq 0$. 
We use a linear splitting rule: say whenever $a_j^n\geq 0$, 
the mass $\rho_j^n \times a_j^n \Delta t/\Delta x$ is sent to grid point $x_{j+1}$ whereas $\rho_j^n \times 
(1 - a_j^n \Delta t/\Delta x)$ is sent to grid point $x_{j}$. We let the reader verify that this gives the scheme defined above. 
\end{itemize}
\end{remark}

\subsection{Properties of the scheme}

Throughout the analysis, $a_\infty$ stands for $|| a ||_{L^\infty(\RR \times \RR^d)}$
whenever $a \in L^\infty([0,+\infty),L^\infty(\RR^d))^d$. 
We will assume further that $\Delta x \leq 1$.
\vspace{5pt}
\\
The following lemma states a Courant-Friedrichs-Lewy-like (CFL) condition ensuring that the scheme 
preserves nonnegativity:
\begin{lemma}\label{lem:CFL}
Let $a\in L^\infty([0,+\infty),L^\infty(\RR^d))^d$
and let $(\rho_{J}^0)_{J\in \ZZ^d}$ be defined by \eqref{disrho0} for some
$\rho^{ini}\in \calP_p(\RR^d)$, $p\geq 1$.
Assume further that the following CFL condition holds:
\begin{equation}\label{CFL}
a_\infty \sum_{i=1}^d \frac{\Delta t}{\Delta x_i} \leq 1.
\end{equation}
Then the sequence $(\rho_J^n)_{n,J}$ computed thanks to the scheme 
defined in \eqref{dis_num}--\eqref{def:ai}
is non-negative: for all $J\in \ZZ^d$ and $n\in \NN$,
$\rho_{J}^n \geq 0$.
\end{lemma}
\begin{proof}

We can rewrite equation \eqref{dis_num} as
\begin{equation}
\begin{array}{ll}
\ds \rho_{J}^{n+1} = & \ds \rho_{J}^n \left[ 1 - 
\sum_{i=1}^d \frac{\Delta t}{\Delta x_i} \big|{a_i}^n_{J}\big|\right] 
+ \sum_{i=1}^d \rho_{J+e_i}^n \frac{\Delta t}{\Delta x_i}({a_i}^n_{J+e_i})^-
+ \sum_{i=1}^d \rho_{J-e_i}^n \frac{\Delta t}{\Delta x_i}({a_i}^n_{J- e_i})^+.
\label{schemarho}
\end{array}
\end{equation}
From definition \eqref{def:ai}, we have $|{a_i}_{J}^{n}|\leq a_\infty$ for $i=1,\ldots,d$. 
Thus assuming Condition \eqref{CFL}, we deduce that in \eqref{schemarho}
all the coefficients in front of $\rho_{J}^n$, $\rho_{J-e_i}^n$ and $\rho_{J+e_i}^n$,
$i=1,\ldots,d$, are non-negative. By a straightforward induction argument, as $\rho_J^0 \geq 0$ 
for all $J \in \ZZ^d$, $\rho_{J}^{n+1}\geq 0$ for all $J\in \ZZ^d$.
\end{proof}

In the next lemma, we collect some useful properties of the scheme, among which mass conservation
and finiteness of the $p$th order moment: 

\begin{lemma}\label{bounddismom}
Let $a\in L^\infty([0,+\infty),L^\infty(\RR^d))^d$
and let $(\rho_{J}^0)_{J\in \ZZ^d}$ be defined by \eqref{disrho0} for some
$\rho^{ini}\in \calP_p(\RR^d)$, $p\geq 1$.
Let us assume that the CFL condition \eqref{CFL} is satisfied.
Then the sequence $(\rho_{J}^n)_{n \in \NN, J\in \ZZ^d}$ given by the numerical scheme
\eqref{dis_num}--\eqref{def:ai} satisfies:

$(i)$ Conservation of the mass:
for all $n\in \NN^*$, we have
$$
\sum_{J\in \ZZ^d} \rho_{J}^n = \sum_{J\in \ZZ^d} \rho_{J}^0 = 1.
$$

$(ii)$ Bound on the $p$th moment: 
there exists a constant $C_p>0$, only depending on $a_{\infty}$, 
the dimension $d$ and the exponent $p$, such that, for all $n\in \NN$, we have
\[
\begin{array}{l}
\ds M_1^n := \sum_{J\in \ZZ^d} |x_J| \rho_{J}^n \leq M_1^0 + C_1 t^n, \qquad \mbox{ when } p=1,
\vspace{2pt}
\\
\ds M_p^n := \sum_{J\in \ZZ^d} |x_J|^p\rho_{J}^n \leq e^{C_p t^n} \big(M_p^0 + C_p\big), \qquad \mbox{ when } p>1,
\end{array}
\]
where we recall that $t^n=n\Delta t$.

$(iii)$ Support: let us define $\lambda_i = \Delta t / 
\Delta x_{i}$. 
If $\rho^{ini}$ has a bounded support then the numerical approximation at finite time $T$
has a bounded support too.
More precisely, assuming that $\rho^{0}_{\Delta x}$ is compactly supported in $B(0,R)$,
then for any $T \geq 0$, for any integer $n\leq T/\Delta t$, we have 
$$
\rho^n_{J}=0 \quad \mbox{for any } J\in \ZZ^d \mbox{ such that } 
x_J \notin B\Bigl(0,R+\frac{T}{\min_{i=1,\ldots,d} \{\lambda_i\}}\Bigr). 
$$
\end{lemma}

\begin{proof}
According to Lemma \ref{lem:CFL}, the weights $((\rho_{J}^n)_{J \in \ZZ^d})_{n \in \NN}$ are non-negative.

$(i)$ The mass conservation is directly obtained by summing
equation \eqref{dis_num}  over $J$.

$(ii)$ Let $p\geq 1$. By a discrete integration by parts on \eqref{dis_num}, we get
\begin{equation}\label{ineqxp}
\begin{array}{ll}
\ds \sum_{J\in \ZZ^d} |x_J|^p \rho_{J}^{n+1} = & \ds \sum_{J\in \ZZ^d} |x_J|^p\rho_{J}^n - 
 \sum_{i=1}^d\frac{\Delta t}{\Delta x_i} \sum_{J\in \ZZ^d} ({a_i}^n_{J})^+ 
 \, \big(|x_J|^p-|x_{J+e_i}|^p\big) \rho_{J}^n \\[2mm]
& \ds + \sum_{i=1}^d\frac{\Delta t}{\Delta x_i} \sum_{J\in \ZZ^d} 
({a_i}^n_{J})^- \big(|x_{J}|^p-|x_{J-e_i}|^p\big) \rho_{J}^n.
\end{array}
\end{equation}
Letting $c_p = p 2^{p-1}$, one has  
\begin{equation}\label{ineqp}
\Big| |x|^p - |x\pm \Delta x_i e_i|^p \Big| \leq 
c_p \Big( \Delta x_i |x|^{p-1} + \Delta x_i^{p}\Big).
\end{equation}
Indeed, thanks to the convexity of $x \mapsto | x|^p$, we have
\begin{equation}
\label{eq:convexity:bound:moments}
\begin{cases}
\ds |x|^p \geq | x \pm \Delta x_i e_i |^p \mp p | x \pm \Delta x_i e_i |^{p-2} 
\langle x \pm \Delta x_i e_i, \Delta x_i e_i \rangle, 
\\
\ds | x \pm \Delta x_i e_i |^p \geq |x|^p \pm p|x|^{p-2} \langle x, \Delta x_i e_i \rangle.
\end{cases}
\end{equation}
Above, $\vert x \pm \Delta x_{i} e_{i} \vert^{p-2}
\langle x \pm \Delta x_{i} e_{i},\Delta x_{i} 
e_{i} \rangle$ is understood as 
$0$ when $x \pm \Delta x_{i} e_{i}=0$, 
and similarly for 
$|x|^{p-2} \langle x, \Delta x_i e_i \rangle$.

Now, the first line in 
\eqref{eq:convexity:bound:moments}, yields
$$
| x \pm \Delta x_i e_i |^p - |x|^p \leq p \Delta x_i | x \pm \Delta x_i e_i |^{p-1} \leq 
p 2^{p-1}\left( \Delta x_i |x|^{p-1} + \Delta x_i^p\right) 
$$
 (actually true with $2^{p-2}$ instead of $2^{p-1}$), whilst the second line gives
$$
| x \pm \Delta x_i e_i |^p - |x|^p \geq - p |x|^{p-1} \Delta x_i \geq 
- p 2^{p-1}\left( \Delta x_i |x|^{p-1} + \Delta x_i^p\right). 
$$
We easily get 
\eqref{ineqp}.

Then, using inequality \eqref{ineqp}
together with the mass conservation and the fact that $\Delta x \leq 1$, we deduce from \eqref{ineqxp}:
$$
\begin{array}{ll}
\ds \sum_{J\in \ZZ^d} |x_J|^p \rho_{J}^{n+1} & \ds \leq \sum_{J\in \ZZ^d}
|x_J|^p \rho_{J}^{n} + c_p\Delta t \sum_{i=1}^d\sum_{J\in \ZZ^d} 
\Big(|x_J|^{p-1}+ \Delta x_i^{p-1}\Big)\, \rho_{J}^n \, |{a_i}^n_{J}| \\[5mm]
& \ds \leq \sum_{J\in \ZZ^d} |x_J|^p \rho_{J}^{n} + c_p d \Delta t a_\infty 
\biggl(\sum_{J\in \ZZ^d} |x_J|^{p-1}\, \rho_{J}^n + 
{1}
\biggr),
\end{array}
$$
which may be rewritten as
\begin{equation*}
M^{n+1}_{p} \leq M_{p}^n + c_{p} d \Delta t 
a_{\infty}
\bigl( M^{n}_{p-1} + 1 \bigr),
\end{equation*}
where, by mass conservation, $M^n_{0} = 1$.
For $p = 1$, 
we get the result
with $C_1 = 2c_p d a_\infty = 2d a_\infty$ by a
 straightforward induction. For $p > 1$, 
thanks to H\"older's inequality, $M_{p-1}^n \leq (M_p^n)^{(p-1)/p} 
\leq (M_p^n + 1)^{(p-1)/p} \leq M_p^n + 1$ and, we get 
\begin{equation*}
\begin{split}
M_p^{n+1} 
&\leq  M^n_{p} + c_{p} d \Delta t a_{\infty} \bigl( M^n_{p} + 2
\bigr)
\\
&\leq (1 + C_p\Delta t) M_p^n + C_p\Delta t
\end{split}
\end{equation*}
with $C_p = 2 c_p d a_\infty$. 
We conclude the proof using a discrete Gr\"onwall lemma.

$(iii)$ By definition of the numerical scheme \eqref{dis_num}, we clearly have that 
the support increases of only one cell in each direction at each time step. 
Therefore, after $n$ {iterations}, the support has increased of less than
$$
n \times \max_{i=1,\ldots,d} \{\Delta x_i\} = \frac{n\, \Delta t}{\min_{i=1\ldots,d}\{\lambda_i\}} 
\leq \frac{T}{\min_{i=1,\ldots,d}\{\lambda_i\}},
$$
provided that $T \geq n \Delta t$. 
\end{proof}

\subsection{Probabilistic interpretation}\label{sec:proba}

Following the idea in \cite{caniveau}, we associate 
random characteristics with the here above upwind scheme. The construction of these 
characteristics is based upon the trajectories of a Markov chain admitting $\ZZ^d$ as state space. 
Here
the random characteristics will be {\em forward} characteristics 
whilst they are {\em backward} 
in \cite {caniveau}.
The rationale for considering forward characteristics 
lies in the fact that the equation 
is conservative; subsequently, the expression of the solution
provided by Theorem \ref{thPR}
is based upon a forward flow. On the opposite, 
the equation
considered in \cite{caniveau}
is of the non-conservative form 
 $\partial_t \rho + a \cdot \nabla \rho = 0$
 and the expression of the solution involves  backward characteristics.

Throughout the analysis, we will denote by $\Omega= (\ZZ^d)^{\mathbb N}$ the canonical 
space for the Markov chain.
The canonical process is denoted by $(K^{n})_{n \in {\mathbb N}}$ (namely $K^n$ maps 
$\omega = (\omega^n)_{n \in \NN} \in \Omega$ onto the $n^{\textrm{th}}$ 
coordinate $\omega^n$ of $\omega$): 
$K^n$ must be understood as the $n^{\textrm{th}}$ site occupied by a random 
process taking values in $\ZZ^d$.
Notice that we here adopt a non-standard notation for the time index as we put it in superscript 
instead of subscript; although it does not fit
the common habit, we feel it more consistent with the notation used
above for defining the numerical scheme. 

We equip $\Omega$ with the standard Kolmogorov $\sigma$-field 
${\mathcal A}$ generated by sets of the form $\prod_{n \in {\mathbb N}} A^n$, 
with $A^n \subset \ZZ^d$ for all $n \in {\mathbb N}$
and, for some integer $n_{0} \geq {0}$, $A^n=\ZZ^d$ for $n \geq n_{0}$. In 
other words, ${\mathcal A}$ is the 
smallest $\sigma$-field that renders each $K^{n}$, $n \in {\mathbb N}$, 
measurable. Indeed, 
for any integer $n_{0} \geq 0$ and any subsets $A^0,\dots,A^{n_{0}}
\subset \ZZ^d$, the pre-image 
$(K^{0},\dots,K^{n_{0}})^{-1}(A^0 \times \dots  \times A^{n_{0}})$
is precisely the cylinder $A^0 \times \dots \times A^{n_{0}} \times 
\prod_{n > n_{0}} \ZZ^d$. 
The canonical filtration generated by $(K^n)_{n \in {\mathbb N}}$
is denoted by ${\mathbb F}=({\mathcal F}^{n}= 
\sigma(K^{0},\cdots,K^{n}))_{n \in {\mathbb N}}$.
For each $n \geq 0$, 
${\mathcal F}^n$ is the sub-$\sigma$-field of ${\mathcal A}$
containing events of the form $A^{(n)} \times \prod_{k>n}
\ZZ^d$, with $A^{(n)} \subset (\ZZ^d)^{n+1}$. Informally, ${\mathcal F}^{n}$ stands 
for the information that an observer would collect by observing the random 
characteristic up until time $n$ (or equivalently the realizations of $K^{0},\dots,K^{n}$). 

We then endow the pair $(\Omega,{\mathcal A})$ with a collection of probability 
measures $({\mathbb P}_{\mu})_{\mu \in {\mathcal P}(\ZZ^d)}$, ${\mathcal P}(\ZZ^d)$ denoting the 
set of probability measures on $\ZZ^d$,
such that, for all $\mu \in {\mathcal P}(\ZZ^d)$, $(K^{n})_{n \in {\mathbb N} }$ is a 
time-inhomogeneous Markov chain under 
${\mathbb P}_{\mu}$ 
with initial law $\mu$, namely $K^{0}{}_{\#}{\mathbb P}_{\mu} = \mu$ (\textit{i.e.}
${\mathbb P}_{\mu}(K_{0}=J) = \mu(J)$, 
which means that the initial starting cell is picked 
at random according to the law $\mu$; sometimes, we will also write $K^{0}\sim \mu$),
and with transition matrix at time $n \geq 0$:
\begin{equation}
\label{defKj}
P_{J,L}^n =
\left\{\begin{array}{ll}
\ds 1-\sum_{i=1}^d \frac{\Delta t}{\Delta x_i} 
\big\vert {a_i}^n_{J}
\big\vert 
&\textrm{when} \ L=J,
\vspace{5pt}
\\
\ds \frac{\Delta t}{\Delta x_i}({a_i}^n_{J})^+ 
&\textrm{when} \ L=
J + e_i, \ \ \textrm{for} \ i= 1,\dots,d,
\vspace{5pt}
\\ 
\ds \frac{\Delta t}{\Delta x_i}({a_i}^n_{J})^- 
&\textrm{when} \ L=
J - e_i, \ \ \textrm{for} \ i=1,\dots,d,
\vspace{5pt}
\\
0 \quad &\textrm{otherwise},
\end{array}\right.
\end{equation}
where we used 
\eqref{def:ai}
under the assumption that 
$a\in L^\infty([0,+\infty),L^\infty(\RR^d))^d$
satisfies the CFL condition 
\eqref{CFL} (we assume it to be in force throughout the 
section). 
\vspace{5pt}

For any $\mu \in {\mathcal P}(\ZZ^d)$,
we write $\EE_{\mu}$ for the expectation under $\PP_{\mu}$. 
Also, for any $\mu \in {\mathcal P}(\ZZ^d)$
and any $n \in {\mathbb N}$, the conditional probability $\PP_{\mu}( \cdot \vert {\mathcal F}^{n})$ 
is just denoted by $\PP_{\mu}^n( \cdot)$; similarly, 
the conditional expectation $\EE_{\mu}(\cdot \vert {\mathcal F}^{n})$ is denoted by $\EE_{\mu}^n(\cdot)$.
Moreover, in statements that are true independently of the initial 
distribution $\mu \in {\mathcal P}(\ZZ^d)$, we often 
drop the index $\mu$ in the symbols $\PP_{\mu}$ and/or 
$\EE_{\mu}$. For instance, we may write:
\begin{equation*}
\forall n \in {\mathbb N},
\ \forall L \in {\mathbb Z}^d, \quad 
\PP^{n} \bigl(K^{n+1}= L \bigr) = P^n_{K^{n},L}.
\end{equation*}

Whenever $\mu$ is the Dirac mass at some $J \in \ZZ^d$, namely 
$\mu = \delta_{J}$, we write $\PP_{J}$ instead of $\PP_{\mu}$
and similarly for $\EE$.
Notice that, for any $\mu \in {\mathcal P}(\ZZ^d)$, 
$\PP_{\mu}$ is entirely determined by $\mu$ and
the collection $(\PP_{J})_{J \in {\mathbb Z}^d}$:
\begin{equation*}
\PP_{\mu}(\cdot) = \sum_{J \in {\mathbb Z}^d} \mu_{J}
\PP_{J}(\cdot).
\end{equation*}
In most cases, it thus suffices to restrict the analysis of the Markov chain to the cases when 
$\mu = \delta_{J}$, for $J \in \ZZ^d$.
\vspace{5pt}

The following lemma gives the connection between the sequence of 
weights $((\rho^n)_{J \in \ZZ^d})_{n \in {\mathbb N}}$
introduced in the previous section (defined by the upwind scheme) 
and the Markov chain with transition matrix $P$:

\begin{lemma}
\label{lem:transport:rho}
Given an initial distribution 
$\rho^0 = (\rho^0_{J})_{J \in \ZZ^d}\in {\mathcal P}(\ZZ^d)$, 
define, for any $n \in {\mathbb N}$, 
$\rho^n = (\rho^n_{J})_{J \in \ZZ^d}$ through 
the scheme 
\eqref{dis_num}, namely
\begin{equation*}
\begin{split}
\ds \rho_{J}^{n+1} &= \rho_{J}^n - \sum_{i=1}^{d} \frac{\Delta t}{\Delta x_i}
\Big(({a_i}^n_{J})^+ \rho_{J}^n - ({a_i}^n_{J+ e_i})^- \rho_{J+e_i}^n 
-({a_i}^n_{J- e_i})^+ \rho_{J-e_i}^n + ({a_i}^n_{J})^- \rho_{J}^n \Big)
\\
&= \rho^n_{J} P_{J,J}^n + \sum_{i=1}^d 
\bigl( \rho^n_{J-e_{i}} P_{J-e_{i},J}^n
+ \rho^n_{J+e_{i}} 
P_{J+e_{i},J}^n \bigr). 
\end{split}
\end{equation*}
Then, for any $n \in {\mathbb N}$,
one has 
$\rho^n = K^n{}_{\#} \PP_{\rho^0}$
(equivalently $\rho^n$ is the law of $K^{n}$ 
when the chain is initialized with $\rho^0$). 
\end{lemma}

\begin{proof}
For any $\mu \in {\mathcal P}(\ZZ^d)$, 
we have 
\begin{equation*}
\begin{split}
K^{n+1}{}_{\#} \PP_{\mu}
= \sum_{L \in \ZZ^d} \delta_{L}
\PP_{\mu} \bigl( K^{n+1} = L \bigr)
= \sum_{L \in \ZZ^d}
\sum_{J \in \ZZ^d} 
\delta_{L}
\PP_{\mu} \bigl( K^{n} = J \bigr) 
P^n_{J,L}.
\end{split}
\end{equation*}
Therefore, 
\begin{equation*}
\begin{split}
K^{n+1}{}_{\#} \PP_{\mu} &= \sum_{L \in \ZZ^d}
\delta_{{L}} \Bigl( \PP_{\mu} \bigl( K^{n} = L \bigr) 
P^n_{L,L}
\\
&\hspace{15pt} +
\sum_{i=1}^d \bigl(
\PP_{\mu} \bigl( K^{n} = L - e_{i}\bigr) 
P_{L-e_{i},L}^n + \PP_{\mu} \bigl( K^{n} = L + e_{i}\bigr) 
P_{L+e_{i},L}^n \bigr) \Bigr).
\end{split}
\end{equation*}
Choosing $\mu = \rho^0$, 
the result follows from a straightforward induction. 
\end{proof}

\vspace{5pt}

Now that we have associated a Markov chain with the weights involved 
in the definition of the upwind scheme, 
we can define, as announced, the
corresponding random characteristics. 
A random characteristic consists of a sequence of random variables $(X^n)_{n \in {\mathbb N}}$ 
from $(\Omega,{\mathcal A})$ into $\RR^d$:
\begin{equation}
\label{defXj}
\forall n \in {\mathbb N},
\
\forall \omega \in \Omega,
\quad 
X^n(\omega) = x_{K^n(\omega)},
\end{equation}
where we recall that 
$x_J = (J_1\Delta x_1, \ldots, J_d \Delta x_d)$
whenever 
$J = (J_{1},\dots,J_{d}) \in {\mathbb Z}^d$. 

\begin{proposition}\label{propZ}
Let $X^n$ be the random variable defined by \eqref{defXj}
through 
the Markov chain admitting 
$P$ in \eqref{defKj} as transition matrix.

(i) For all $J\in \ZZ^d$, we have, with probability one under
$\PP_{J}$,
\begin{equation}\label{en}
\E^n_{J}(X^{n+1}-X^n) = \int_{t^n}^{t^{n+1}} 
a(s,X^n)
\,ds.
\end{equation}

(ii) Defining $\ds \rho^n_{\Delta x} = \sum_{J\in \ZZ^d} \rho_{J}^n \delta_{x_J}$,
we have $\rho^{n}_{\Delta x} = X^n{}_\# \PP_{\rho_{\Delta x}^0}$.

\end{proposition}
\begin{proof}
(i) 
For $J \in \ZZ^d$, 
we compute the conditional expectation given the trajectory of the Markov chain up until time $n$:
$$
\E^n_{J}(X^{n+1}-X^n) = \sum_{i=1}^d \left(\Delta x_i e_i
({a_i}^n_{K^n})^+\frac{\Delta t}{\Delta x_i}
- \Delta x_i e_i ({a_i}^n_{K^n})^- \frac{\Delta t}{\Delta x_i}\right).
$$
We deduce \eqref{en} by using definition \eqref{def:ai}.

(ii)
For any $\mu \in {\mathcal P}(\ZZ^d)$, 
we have, for all $n \in {\mathbb N}$, 
\begin{equation*}
\begin{split}
X^{n}{}_{\#} \PP_{\mu}
&= \sum_{L \in \ZZ^d} 
\delta_{x_{L}}
\PP_{\mu} \bigl( K^{n} = L \bigr)
= \sum_{L \in \ZZ^d} 
\delta_{x_{L}}
\bigl( K^{n}{}_{\#}
\PP_{\mu} 
\bigr)_{L}.
\end{split}
\end{equation*}
The claim follows from Lemma 
\ref{lem:transport:rho}.
\end{proof}

\section{Order of convergence}
\label{sec:ordre}

This section is devoted to the proof of the main result of our paper, that is 
the $1/2$ order of convergence of the numerical approximation constructed by
the upwind scheme \eqref{dis_num}.
The precise statement of the result is:
\begin{theorem}\label{TH}
Let $\rho^{ini} \in \calP_p(\RR^d)$, for 
some $p\geq 1$.
Let us assume that $a\in L^\infty([0,\infty),L^\infty(\RR^d))^d$ and satisfies the 
one-sided Lipschitz continuity condition \eqref{eq:OSL}.
Let $\rho=Z_\# \rho^{ini}$ be the unique measure solution in the sense of Poupaud and Rascle to the 
conservative transport equation \eqref{eq:transp} with initial datum $\rho^{ini}$
given by Theorem \ref{thPR}.
Let us define
$$
\rho_{\Delta x}^n = \sum_{J\in \ZZ^d} \rho_J^n \delta_{x_J},
$$
where the approximation sequence $((\rho_J^n)_{J \in \ZZ^d})_{n \in \NN}$ is 
computed thanks to the scheme 
\eqref{disrho0}--\eqref{dis_num}--\eqref{def:ai}.
We assume that the CFL condition \eqref{CFL} holds. 
Then, there exists a non-negative constant $C$, depending upon $p$, $\rho^{ini}$ and 
$a_\infty$ only, such that, for all $n\in \NN^*$, 
$$
W_p\bigl(\rho(t^n),\rho_{\Delta x}^n\bigr) \leq C\,e^{2 \int_0^{t^n} \alpha(s)\,ds}\,
\bigl( \sqrt{t^{n} \Delta x} + \Delta x \bigr). $$
\end{theorem}

The proof of this theorem is postponed to Section \ref{sec:proof}. 
We need first 
to establish some useful estimates 
on the distance between the Filippov characteristics 
generated by a one sided Lipschitz continuous velocity field and the approximated characteristics.

\subsection{Approximation of the flow}

In the following lemma, we provide an estimate for the distance between the exact Filippov flow and
an approximating flow computed through an explicit Euler discretization:

\begin{lemma}\label{lem:ZY}
Let $a\in L^\infty([0,\infty),L^\infty(\RR^d))$ satisfying Condition \eqref{eq:OSL}. 
Let us consider $Z$ the Filippov flow associated to the velocity field $a$ and
define by induction
$$
Y^{n+1} = Y^n + \int_{t^n}^{t^{n+1}} a(s,Y^n)\,ds,
$$
with the initial condition $Y^0=Z(0)$.
There exists a 
universal
constant $C$ such that, for all $n\in \NN$,
$$
|Y^{n} - Z(t^n)| \leq C
a_{\infty}
e^{2\int_0^{t^n} \alpha(s)\,ds} \sqrt{t^n \Delta t}.
$$
\end{lemma}

\begin{remark}
The order $1/2$ in this estimate may not be optimal, but it is sufficient for our purpose. 
\end{remark}

\begin{proof}
For a given value of $n \in \NN$,
let us define, for $t \in [t^n, t^{n+1}]$, $Y(t) = Y^n + \int_{t^n}^t a(s,Y^n)\, ds$. 
By definition of the characteristics, we have, for any $t \in [t^n, t^{n+1}]$, 
$$
|Y(t) - Z(t)|^2 = \left|Y^n -Z(t^n) + \int_{t^n}^{t} \bigl(a(s,Y^n)-a(s,Z(s)) \bigr)\,ds\right|^2.
$$
Expanding the right hand side, we get 
\begin{equation}
\label{eq:Y(t)-Z(t)}
\begin{split}
\ds |Y(t) - Z(t)|^2 &\leq |Y^n -Z(t^n)|^2 + 2\int_{t^n}^{t} \big\langle Y^n-Z(s),a(s,Y^n)-a(s,Z(s))\bigr\rangle\,ds 
\\
\hspace{15pt}
& +2 \int_{t^n}^{t} \big\langle Z(s)-Z(t^n),a(s,Y^n)-a(s,Z(s))\bigr\rangle\,ds 
+(2 a_\infty \Delta t)^2.
\end{split}
\end{equation}
Using condition \eqref{eq:OSL}, we deduce
$$
\begin{array}{ll}
\ds |Y(t) - Z(t)|^2 \leq &\ds 
|Y^n -Z(t^n)|^2 + 2\int_{t^n}^{t} \alpha(s) |Y^n-Z(s)|^2\,ds \\[2mm]
& \ds + 2\int_{t^n}^{t} \big\langle Z(s)-Z(t^n),a(s,Y^n)-a(s,Z(s))\bigr\rangle\,ds+
(2 a_\infty \Delta t)^2.
\end{array}
$$
Moreover, since the field $a$ is bounded, we have
$|Z(s)-Z(t^n)|\leq a_\infty |s-t^n|$
and so 
$\int_{t^n}^t |Z(s)-Z(t^n)| ds \leq a_{\infty} (t-t^n)^2/2$.
Thus, 
$$
\ds |Y(t) - Z(t)|^2 \leq  
|Y^n -Z(t^n)|^2 + 2\int_{t^n}^{t} \alpha(s) |Y^n-Z(s)|^2\,ds + 6 a_\infty^2  \Delta t^2, 
$$
and, as we also have 
$\vert Y(s) - Y^n \vert \leq a_{\infty} \vert s - t^n \vert$, 
\begin{equation}
\label{eq:Y(t)-Z(t):2}
\begin{split}
&|Y(t) - Z(t)|^2 
\\
&\hspace{15pt} \leq  
|Y^n -Z(t^n)|^2 + 3 \int_{t^n}^{t} \alpha(s) |Y(s)-Z(s)|^2\,ds + 6 a_\infty^2  \Delta t^2 
+ 6 a_\infty^2 \Delta t^2 \int_{t^n}^t \alpha(s)\, ds,
\end{split}
\end{equation}
where we used the standard Young inequality 
$(|Y(s)-Z(s)| + |Y^n-Y(s)|)^2 \leq (3/2) 
|Y(s)-Z(s)|^2 + 
3 |Y^n-Y(s)|^2$.

Thanks to a continuous Gr\"onwall lemma, the two characteristics thus satisfy 
\[
|Y^{n+1}-Z(t^{n+1})|^2 \leq \biggl[ |Y^n -Z(t^n)|^2 + 6 a_{\infty}^2 \Delta t^2 \biggl(1 
+ \int_{t^n}^{t^{n+1}} \alpha(s)\, ds \biggr) \biggr] e^{3 \int_{t^n}^{t^{n+1}} \alpha(s) \,ds}.
\]
As  $Y^0 = Z(0)$, a discrete Gr\"onwall lemma leads to 
\[
|Y^{n}-Z(t^{n})|^2 \leq  6 a_{\infty}^2 n \Delta t^2 \left( 1
+ \int_{0}^{t^{n}} \alpha(s)\, ds \right) e^{3 \int_{0}^{t^{n}} \alpha(s) \,ds}, 
\]
which can also be replaced by the maybe simpler estimate 
\[
|Y^{n}-Z(t^{n})|^2 \leq 
C a_{\infty}^2 n \Delta t^2 e^{4\int_{0}^{t^{n}} \alpha(s) \,ds}.
\]
This completes the proof.
\end{proof}

\begin{remark}
\label{rem:a:Lip}
Whenever $a$ is $L$-Lipschitz in time (uniformly in space) and 
$a(s,Y^n)$ is replaced by $a(t^n,Y^n)$ in the recursive definition of the sequence $(Y^n)_{n \in \NN}$ in the statement of Lemma 
\ref{lem:ZY}, there is an additional term in 
\eqref{eq:Y(t)-Z(t)}, coming from the time discretization of 
the velocity field. This term has the form:
\begin{equation}
\label{eq:a:Lip:2}
2 \int_{t^n}^t \Bigl\langle Y^n - Z(s), a (t^n,Y^n) - a(s,Y^n)
\Bigr\rangle ds,
\end{equation}
which is less than $2 L \Delta t \int_{t^n}^t \vert Y^n - Z(s) \vert ds $. By Young's inequality, we get 
\begin{equation*}
\begin{split}
2 \int_{t^n}^t \Bigl\langle Y^n - Z(s), a (s,Y^n) - a(t^n,Y^n)
\Bigr\rangle ds
&\leq L \Delta t \int_{t^n}^t \vert Y^n - Z(s) \vert^2 ds
+ L \Delta t^2. 
\end{split}
\end{equation*}
This gives a similar inequality 
to \eqref{eq:Y(t)-Z(t):2} but with $\alpha$ replaced by 
$\alpha + L \Delta t$ and $a_{\infty}$ replaced by 
$a_{\infty} + L$. 
The corresponding version of Lemma 
\ref{lem:ZY}
is easily 
derived.

Alternatively, we may
perform all the above computations with respect to the OSL constant 
of $a(t^n,\cdot)$ instead of $a(s,\cdot)$. Instead of 
\eqref{eq:a:Lip:2}, we then focus on 
\begin{equation*}
2 \int_{t^n}^t \Bigl\langle Y^n - Z(s), a (t^n,Z(s)) - a(s,Z(s))
\Bigr\rangle ds.
\end{equation*}
Then we obtain the same conclusion, 
with 
$\alpha$ replaced by 
$\alpha + L \Delta t$ and $a_{\infty}$ replaced by 
$a_{\infty} + L$, but also the integral $\int_{0}^{t_{n}}
\alpha(s) ds$ in the statement has to be replaced 
by the Riemann sum $\Delta t \sum_{k=0}^{n-1} \alpha(t^k)$.  
\end{remark}

\subsection{Distance between the Euler scheme and the random characteristics}

\begin{lemma}\label{lem:estim1}
Under the CFL condition \eqref{CFL}, consider 
the random characteristics 
$(X^n)_{n \in {\mathbb N}}$ defined in \eqref{defXj}.
Then, for any initial condition $\mu \in {\mathcal P}(\ZZ^d)$, 
it holds that, with probability $1$ under $\PP_{\mu}$,
for all $n \in {\mathbb N}$, 
\begin{equation}\label{Xijn}
X^{n+1} = X^n+ \int_{t^n}^{t^{n+1}} 
a(s,X^n) \,ds + h^n,
\end{equation}
where $h^n$ is an ${\mathcal F}^{n+1}$-measurable $\RR^d$-valued random variable that satisfies
\begin{equation}\label{eq:hn}
\begin{split}
&\E^n_{\mu}(h^n)=0\, ; \quad 
\vert h^n \vert \leq 2 \Delta x\, ; \quad \forall p \geq 1, \quad
\E^n_{\mu} \bigl( \vert h^n \vert^p
\bigr) \leq 2^p C  {a}_{\infty}
\Delta t \Delta x^{p-1},
\end{split}
\end{equation}
for a constant $C$ that depends only on $d$.

In particular, if we define
iteratively the following sequence
of random variables $(\hat{Y}^n)_{n \in {\mathbb N}}$
(constructed on the space $(\Omega,{\mathcal A})$
that supports the random characteristics):
\[
\hat{Y}^{n+1}=\hat{Y}^{n} + \int_{t^n}^{t^{n+1}} a(s,\hat{Y}^n)\,ds,
\]
with the (random) initial datum $\hat{Y}^0 = X^0 = x_{K^0}$, 
then, provided that $\Delta x \leq 1$, there exists, for any $p \geq 1$, a non-negative constant $C_p$, only depending on $p$, $d$ and $a_{\infty}$, such that
\begin{equation}
\label{eq:estim1:final}
\forall n \in {\mathbb N}, \quad
\E_{\mu}\bigl(|X^{n} - \hat{Y}^{n}|^p\bigr)^{1/p} \leq 
C_p \, e^{ \int_0^{t^n} \alpha(s)\,ds}
\bigl( 
\sqrt{t^n \Delta x}
+ \Delta 
x 
\bigr).
\end{equation}
\end{lemma}

\begin{proof}
The 
expansion
\eqref{Xijn}, 
with 
$h^n$ satisfying 
$\E^n_{\mu}(h^n)=0$
for each $n \in {\mathbb N}$,
is a direct consequence 
of \textit{(i)} in Proposition \ref{propZ}. 
By construction, we can write 
$h^n = X^{n+1} - X^n - \EE^n_{\mu}(X^{n+1}-X^n)$. 
Since $\vert X^{n+1} - X^n \vert \leq \Delta x$, 
we deduce that 
$\vert h^n \vert \leq 2 \Delta x$. 
Moreover, for any $p \geq 1$,
\begin{equation*}
\E^n_{\mu}
\bigl( \vert h^n \vert ^p
\bigr) \leq 2^{p-1} \Bigl[ \E^n_{\mu}
\bigl( \vert X^{n+1} - X^n \vert ^p
\bigr) +
\E^n_{\mu}
\bigl( \vert \EE^{n}_{\mu} (X^{n+1} - X^n) \vert ^p
\bigr)
\Bigr]
\leq 2^p \E^n_{\mu}
\bigl( \vert X^{n+1} - X^n \vert ^p
\bigr).
\end{equation*}
Then, the bound for $\E^n_{\mu}(\vert h^n \vert^p)$ follows from the fact that:
\begin{equation*}
\begin{split}
{\mathbb E}_{\mu}^n 
\bigl(
\vert X^{n+1} - X^n \vert^p
\bigr)
& =
{\mathbb E}_{\mu}^n 
\bigl(
\vert X^{n+1} - X^n \vert^{p}
{\mathbf 1}_{\{K^{n+1} \not = K^{n}\}}
\bigr)
\\
&\leq \Delta x^{p} \PP^{n}_{\mu}
\bigl( K^{n+1} \not = K^{n} \bigr)
\leq \Delta x^{p} \sum_{i=1}^d
\bigl(
P^n_{K^{n},K^{n}+e_{i}}
+
P^n_{K^{n},K^{n}-e_{i}}
\bigr)
\leq C \Delta x^{p-1} \Delta t
\end{split}
\end{equation*}
with $C = da_\infty$, for all $p \geq 1$. 

We split the proof of 
the second claim
\eqref{eq:estim1:final}, 

into two steps. In the first step, we will estimate $\E(|X^n-\hat{Y}^n|^p)^{1/p}$, for $p \in [1,2]$.
The second step is devoted to the analysis of $\E(|X^n-\hat{Y}^n|^p)^{1/p}$ when $p > 2$. 

{\bf First step.}
From definition \eqref{Xijn}, 
we obtain, after an obvious expansion, 
\begin{equation}
\label{eq:a:Xn+1:Yn+1}
\begin{split}
\ds |X^{n+1} - \hat{Y}^{n+1}|^2 &\leq |X^{n} - \hat{Y}^{n}|^2 
+ 2\int_{t^n}^{t^{n+1}} 
\Big\langle X^{n} - \hat{Y}^{n}, a(s,X^{n}) 
- a(s,\hat{Y}^n)
\Big\rangle\,ds 
\\[2mm]
&\hspace{15pt}\ds + 2\langle X^{n} -\hat{Y}^{n},h^n \rangle 
+ \vert h^n \vert^2
+  4 a_{\infty} \bigl( a_{\infty} \Delta t^2 + \Delta t \Delta x
\bigr),
\end{split}
\end{equation}
using the fact that $a$ is bounded and that $|h^n| \leq 2\Delta x$, see \eqref{eq:hn}.

Using the CFL condition 
\eqref{CFL} in order to bound $a_{\infty} \Delta t^2$ by 
$\Delta t \Delta x/d$, we get
\begin{equation*}
\begin{split}
\ds |X^{n+1} - \hat{Y}^{n+1}|^2 &\leq |X^{n} - \hat{Y}^{n}|^2 
+ 2\int_{t^n}^{t^{n+1}} 
\Big\langle X^{n} - \hat{Y}^{n}, a(s,X^{n}) 
- a(s,\hat{Y}^n)
\Big\rangle\,ds 
\\[2mm]
&\hspace{15pt}\ds + 2\langle X^{n} -\hat{Y}^{n},h^n \rangle 
+ \vert h^n \vert^2
+ C a_{\infty}  \Delta t \Delta x,
\end{split}
\end{equation*}
for a constant $C$ that only depends on $d$ and whose value is allowed to increase 
from line to line. Since $a$ satisfies the one-sided Lipschitz continuity condition \eqref{eq:OSL}, 
we get that, 
with probability 1 under $\PP_{\mu}$,
for all $n \in {\mathbb N}$,
\begin{equation}
\label{eqXY}
\begin{split}
&|X^{n+1} - \hat{Y}^{n+1}|^2 
\\
&\hspace{15pt} \leq \biggl(1+2 \int_{t^n}^{t^{n+1}}\alpha(s)\,ds\biggr) |X^{n} - \hat{Y}^{n}|^2
+ 2\langle X^{n}-\hat{Y}^{n},h^n\rangle 
+ \vert h^n \vert^2
+ C a_{\infty} \Delta t \Delta x.
\end{split}
\end{equation}
Recalling that $\EE^n_{\mu}(h^n)= 0$
and noticing that $X^n - \hat{Y}^n$ is ${\mathcal F}^{n}$
is measurable, we have 
\begin{equation}\label{eq:Enul}
\EE^n_{\mu}\bigl(\langle X^{n}-\hat{Y}^{n},h^n\rangle\bigr) = 
\bigl\langle X^{n}-\hat{Y}^{n},\EE^n_{\mu}(h^n)
\bigr\rangle = 0. 
\end{equation}
Now taking the conditional expectation $\EE^n_{\mu}$ in \eqref{eqXY} and 
recalling from the preliminary step of the proof that 
$\EE^n_{\mu}(\vert h^n\vert^2) 
\leq C a_{\infty} \Delta t\Delta x$
(with $C = 4d$), 
we obtain 
\[
\EE^n_{\mu}\bigl(|X^{n+1} - \hat{Y}^{n+1}|^2\bigr) \leq \biggl(1+2 \int_{t^n}^{t^{n+1}}\alpha(s)\,ds\biggr) 
|X^{n} - \hat{Y}^{n}|^2 
+ C a_{\infty} \Delta t \Delta x,
\]
for a new value of $C$. 

Taking the expectation $\EE_{\mu}$
(using the tower property $\EE_{\mu}[\cdot]=\EE_{\mu}[\EE_{\mu}^n(\cdot)]$), 
applying a discrete version of Gr\"onwall's lemma, and using also the fact that the initial datum 
verify $X^0=\hat{Y}^0$, we deduce
\begin{equation}\label{ineqX2}
\forall n \in {\mathbb N},
\quad
\E_{\mu}\bigl(|X^n - \hat{Y}^{n}|^2\bigr) \leq 
C a_{\infty} e^{2\int_0^{t^n} \alpha(s)\,ds}
n\Delta t \Delta x.
\end{equation}

Therefore, for $p\in [1,2]$, we have, thanks to H\"older's inequality, 
\begin{equation}\label{ineqX1}
\left( \E_{\mu}\bigl(|X^{n}-\hat{Y}^{n}|^p\bigr)\right)^{1/p} \leq 
\left( \E_{\mu}\bigl(|X^{n}-\hat{Y}^{n}|^2\bigr) \right)^{1/2} \leq e^{\int_0^{t^n} \alpha(s)\,ds} \, \sqrt{C a_{\infty} t^n \Delta x},
\end{equation}
which concludes the proof when $p\in [1,2]$.

{
{\bf Second step.}
In order to handle the case $p \geq 2$, we use an induction.
We assume that, for some $p\in \NN \setminus \{0,1\}$,
there exists a 
constant $c$, only depending on $p$, $d$ and $a_{\infty}$, such that, for all $1 \leq m\leq 2(p-1)$, 
for all $n \in \NN$,
\begin{equation}\label{hyprec}
\E_{\mu}(|X^n-\hat{Y}^n|^m)^{1/m} \leq c e^{\int_0^{t^n} \alpha(s)\,ds} 
\Bigl( \sqrt{t^{n} \Delta x} + \Delta x
\Bigr),
\end{equation}
which is obviously true when $p=2$ thanks to 
\eqref{ineqX1}.
From \eqref{eqXY}, we get
\begin{equation*}
|X^{n+1} - \hat{Y}^{n+1}|^{2p} 
\leq \left( e^{2\int_{t^n}^{t^{n+1}}\alpha(s)\,ds} |X^{n} - \hat{Y}^{n}|^2
+ 2\langle X^{n} -\hat{Y}^{n},h^n \rangle 
+
\vert h^n \vert^2
+ C a_{\infty} \Delta t \Delta x  \right)^p.
\end{equation*}
Expanding the right hand side, we obtain
\begin{equation}
\label{eq:cas:p>=2}
\begin{split}
&|X^{n+1} - \hat{Y}^{n+1}|^{2p} 
\leq e^{2p\int_{t^n}^{t^{n+1}}\alpha(s)\,ds} |X^{n} - \hat{Y}^{n}|^{2p} 
\\[2mm]
&\qquad + p e^{2(p-1) \int_{t^n}^{t^{n+1}}\alpha(s)\,ds} |X^{n}- \hat{Y}^{n}|^{2(p-1)} 
\Big(2\langle X^{n}-\hat{Y}^{n},h^n\rangle 
+ 
 \vert h^n \vert^2
+ C a_{\infty} \Delta t \Delta x  \Big) 
\\
&\qquad\ds + \sum_{k=2}^p \Big(^p_k\Big) e^{2(p-k) \int_{t^n}^{t^{n+1}}\alpha(s)\,ds} |X^n -\hat{Y}^n|^{2(p-k)}
\Big(2\langle X^{n}-\hat{Y}^{n},h^n\rangle 
+ 
\vert h^n \vert^2
+ C a_{\infty} \Delta t \Delta x  \Big)^k.
\end{split}
\end{equation}
By the same token as in \eqref{eq:Enul}, notice that 
(the constant $C$ being allowed to increase from line to line as long as it only depends on $d$)
\begin{equation*}
\begin{split}
&\E_{\mu}^n \biggl[|X^{n} - \hat{Y}^{n}|^{2(p-1)}
\Big(2\langle X^{n}-\hat{Y}^{n},h^n\rangle 
+ 
 \vert h^n \vert^2
+ C  a_{\infty} \Delta t \Delta x  \Big) \biggr] 
\\
&\hspace{5pt} 
= \E_{\mu}^n \biggl[|X^{n} - \hat{Y}^{n}|^{2(p-1)}
\Big( C a_{\infty} \Delta t \Delta x + 
\vert h^n \vert^2 \Big) \biggr] 
\leq C a_{\infty} \Delta t \Delta x \, |X^{n} - \hat{Y}^{n}|^{2(p-1)},
\end{split}
\end{equation*}
where we used the last estimate of \eqref{eq:hn} for the last inequality.

We proceed in a similar way with the last term in 
\eqref{eq:cas:p>=2}. 
Allowing the constant $C$ to depend upon $p$, we have, for all 
$k \in \{2,\dots,p\}$,
\begin{equation*}
\begin{split}
&\E_{\mu}^n
\biggl[
|X^{n} - \hat{Y}^{n}|^{2(p-k)}
\Bigl( 2 \langle X^{n}-\hat{Y}^{n},h^n\rangle
+ 
\vert h^n \vert^2
 + C {a_{\infty}} \Delta t \Delta x  \Big)^k \biggr]
\\
&\hspace{5pt}
\leq C \E_{\mu}^n \Bigl[ \vert X^n - \hat{Y}^n \vert^{2p-k}
\vert h^n \vert^k + 
\vert X^n - \hat{Y}^n \vert^{2(p-k)}
\vert h^n \vert^{2k} \Bigr] + C a_{\infty}^k \Delta t^k \Delta x^k
\vert X^n - \hat{Y}^n \vert^{2(p-k)}
\\[2mm]
&\hspace{5pt} \leq C \Delta t \Delta x^{k-1} \vert X^n - \hat{Y}^n \vert^{2p-k}
+ C \Delta t \Delta x^{2k-1} \vert X^n - \hat{Y}^n \vert^{2(p-k)},
\end{split}
\end{equation*}
where, once again, we used 
\eqref{eq:hn} together with the CFL condition to pass from the second to the third line. {In the last line, we allowed $C$ to depend on $d$ and $p$, but also on $a_{\infty}$.} 

Returning to 
\eqref{eq:cas:p>=2}
and taking the expectation therein (using the fact that 
$\EE_{\mu}[\cdot] = \EE_{\mu}[\EE_{\mu}^n(\cdot)]$), we finally get that: 
\begin{equation*}
\begin{split}
&\EE_{\mu} \bigl[ |X^{n+1} - \hat{Y}^{n+1}|^{2p}
\bigr] 
\leq e^{2p\int_{t^n}^{t^{n+1}}\alpha(s)\,ds} 
\EE_{\mu} \bigl[|X^{n} - \hat{Y}^{n}|^{2p} \bigr]
\\[2mm]
&\qquad\ds + C 
e^{2 (p-1) \int_{t^n}^{t^{n+1}}\alpha(s)\,ds} 
\sum_{k=2}^{2p} 
\Bigl( 
\EE_{\mu} \bigl[ 
|X^n -\hat{Y}^n|^{2p-k}
\bigr] \Delta t \Delta x^{k-1}
\Bigr).
\end{split}
\end{equation*}
Pay attention that sum above runs from $2$ to $2p$ instead of 
$2$ to $p$ in the original inequality 
\eqref{eq:cas:p>=2}. 
Plugging the induction property \eqref{hyprec}, we get, for all $n \in {\mathbb N}$, 
\begin{equation}
\label{eq:n+1:2p}
\begin{split}
&\EE_{\mu} \bigl[ |X^{n+1} - \hat{Y}^{n+1}|^{2p}
\bigr] 
\\
&\hspace{5pt} \leq e^{2p\int_{t^n}^{t^{n+1}}\alpha(s)\,ds} 
\EE_{\mu} \bigl[|X^{n} - \hat{Y}^{n}|^{2p} \bigr]
+ C
e^{2p \int_{0}^{t^{n+1}}\alpha(s)\,ds} 
\sum_{k=2}^{2p}
\Delta t \Delta x^{k-1}
\Bigl( \sqrt{t^{n} \Delta x}
+ \Delta x \Bigr)^{2p-k}
\\
&\hspace{5pt} \leq e^{2p\int_{t^n}^{t^{n+1}}\alpha(s)\,ds} 
\EE_{\mu} \bigl[|X^{n} - \hat{Y}^{n}|^{2p} \bigr]
\\
&\hspace{30pt} + C
e^{2p \int_{0}^{t^{n+1}}\alpha(s)\,ds} 
\Delta t 
\biggl( 
\Delta x^{2p-1}
+ \sum_{k=2}^{2p}
(t^{n})^{p-k/2}
\Delta x^{p+k/2-1}
\biggr),
\end{split}
\end{equation}
where we used 
the bound
$( \sqrt{t^{n} \Delta x}
+ \Delta x)^{2p-k}
\leq C [ (t^{n})^{p-k/2} \Delta x^{2p-k}
+ \Delta x^{2p-k}]$ in order to pass from the second to the third line. 
Notice now that
\begin{equation*}
\begin{split}
\sum_{k=2}^{2p}
(t^{n})^{p-k/2} \Delta x^{p+k/2-1}
&= \Delta x^{p} \sum_{k=2}^{2p}
\sqrt{t^{n}}^{2p-k} \sqrt{\Delta x}^{k-2}
\\
&\leq \Delta x^{p} \Bigl( \sqrt{t^{n}} + \sqrt{\Delta x}
\Bigr)^{2p-2}
\leq C \Delta x^{p} \bigl( (t^{n})^{p-1} + \Delta x^{p-1} \bigr). 
\end{split}
\end{equation*}
Plugging into \eqref{eq:n+1:2p}, we get 
\begin{equation*}
\begin{split}
&\EE_{\mu} \bigl[ |X^{n+1} - \hat{Y}^{n+1}|^{2p}
\bigr] 
\\
&\hspace{15pt} \leq 
e^{2p\int_{t^n}^{t^{n+1}}\alpha(s)\,ds} 
\EE_{\mu} \bigl[|X^{n} - \hat{Y}^{n}|^{2p} \bigr]
+ C
e^{2p \int_{0}^{t^{n+1}}\alpha(s)\,ds} 
\Delta t 
\Delta x^{p}
\bigl( 
(t^{n})^{p-1}
+ \Delta x^{p-1}
\bigr).
\end{split}
\end{equation*}
Iterating over $n$
and
recalling that $X^0=\hat{Y}^0$ and $n\Delta t=t^n$, we 
deduce that, for all $n \in {\mathbb N}$,
\begin{equation*}
\begin{split}
\EE_{\mu} \bigl[ |X^{n} - \hat{Y}^{n}|^{2p}
\bigr] 
&\leq C
e^{2p \int_{0}^{t^{n}}\alpha(s)\,ds} 
t^{n}
\Delta x^{p}
\bigl( (t^{n})^{p-1} + \Delta x^{p-1}
\bigr).
\end{split}
\end{equation*}
We finally 
obtain, for all $n \in {\mathbb N}$,
\begin{equation*}
\begin{split}
\EE_{\mu} \bigl[ |X^{n} - \hat{Y}^{n}|^{2p}
\bigr] 
&\leq C
e^{2p \int_{0}^{t^{n}}\alpha(s)\,ds} 
\Delta x^p
\bigl( t^{n} + \Delta x
\bigr)^p,
\end{split}
\end{equation*}
where we used Young's inequality 
to bound 
$t^n \Delta x^{2p-1}$ by $C ((t^n)^p \Delta x^p+ \Delta x^{2p})$.
Then,
\begin{equation*}
\begin{split}
\EE_{\mu} \bigl[ |X^{n} - \hat{Y}^{n}|^{2p}
\bigr]^{1/2p} 
&\leq C
e^{ \int_{0}^{t^{n}}\alpha(s)\,ds} 
\Bigl( \sqrt{t^{n} \Delta x} + \Delta x
\Bigr),
\end{split}
\end{equation*}

By H\"older's inequality, 
we conclude that \eqref{hyprec} holds for $2(p-1) \leq m \leq 2p$.
By induction, \eqref{hyprec} is satisfied for all $p\in \NN^*$. 

\vspace{6pt}

{\bf Conclusion.}
Finally, from \eqref{ineqX1}
and
\eqref{hyprec}, we conclude the proof.
}
\end{proof}

\begin{remark}
\label{rem:a:Lip:2}
In full analogy 
with Remark \ref{rem:a:Lip}, we may discuss the case when 
$a$ is $L$-Lipschitz in time (uniformly in space) and 
$a(s,\cdot)$ is replaced by $a(t^n,\cdot)$ in the recursive definitions of the sequences $(X^n)_{n \in \NN}$
and 
$(\hat Y^n)_{n \in \NN}$ in the statement of Lemma 
\ref{lem:ZY}. 
Then, 
the final result is the same provided that the 
integral $\int_{0}^{t_{n}} \alpha(s) ds$ is 
replaced by $\Delta t \sum_{k=0}^{n-1} \alpha(t^k)$. 
\end{remark}

\subsection{Proof of Theorem \ref{TH}}\label{sec:proof}

Let $\rho^n_{\Delta x} = \sum_{J\in \ZZ^d} \rho_{J}^n \delta_{x_J}$ be the measure associated 
to the numerical solution 
given by the scheme \eqref{disrho0}--\eqref{dis_num}--\eqref{def:ai} at time $t^n$. Since 
$\rho^{ini}\in \calP_p(\RR^d)$,
we first notice from Lemma \ref{bounddismom} that 
$\rho_{\Delta x}^n\in \calP_p(\RR^d)$.
By Proposition
\ref{propZ}, we have 
$\rho^n_{\Delta x}=X^n\,_\#\PP_{\rho^0_{\Delta x}}$ where $X^n$ is defined in \eqref{defXj}
and $\rho^0_{\Delta x} = \sum_{J\in \ZZ^d} \rho_{J}^0 \delta_{x_J}$, with $\rho^0$ defined in 
\eqref{disrho0}.
Let $\rho$ be the exact solution of Theorem \ref{thPR}, 
$\rho(t)=Z(t)_\#\rho^{ini}$ where $Z$ is the Filippov flow 
associated to the one-sided Lipschitz continuous velocity field $a$.

Consider now the two sequences $(Y^n)_{n \in {\mathbb N}}$
and $(\hat{Y}^n)_{n \in {\mathbb N}}$ 
respectively defined in Lemmas 
\ref{lem:ZY} and 
\ref{lem:estim1}. Each $Y^n$ is regarded as 
a mapping from $\RR^d$ into itself: the initial condition 
$Y^0$ (also equal to $Z(0)$) in the statement of Lemma 
\ref{lem:ZY} is given by the identity mapping $\RR^d \ni x 
\mapsto x \in \RR^d$, that is $Y^0(x)=x$ for all $x \in \RR^d$. 
We then call $Y^n(x)$ the value of $Y^n$ in the statement of Lemma 
\ref{lem:ZY} when $Y^0(x)=x$.
When $\RR^d$ is equipped with the 
distribution $\rho^0_{\Delta x}$, the distribution of $Y^n$
writes $Y^n{}_{\#} \rho^0_{\Delta x}$.
In comparison with, each $\hat{Y}^n$ is a random variable 
from $\Omega$ to $\RR^d$: when $\hat{Y}^0$ (also equal to 
$X^0$) has the distribution $\rho^0_{\Delta x}$, the distribution of $\hat{Y}^n$ writes 
$\hat{Y}^n{}_{\#} \PP_{\rho^0}$. It is then crucial to observe that
$\hat{Y}^n(\omega)$
may be regarded as 
 $Y^n(X^0(\omega))$. In particular, 
if both $Y^0$ and $\hat{Y}^0$ have 
$\rho^0_{\Delta x}$ as common law (although the mappings are constructed on different spaces), 
then $Y^n$ and $\hat{Y}^n$ also have the same distribution, 
namely $Y^n{}_{\#} \rho^0_{\Delta x}
= \hat{Y}^n{}_{\#} \PP_{\rho^0}$. 

As a consequence of the above discussion, we deduce from
the triangle inequality:
\begin{equation}\label{ineq}
\begin{array}{ll}
\ds W_p\bigl(\rho^n_{\Delta x},\rho(t^n)\bigr) \leq &\ds
W_p\bigl(X^n{}_{\#} \PP_{\rho^0},\hat Y^n{}_\# \PP_{\rho^0}\bigr) + W_p\bigl(Y^n{}_\# 
\rho_{\Delta x}^0,Z(t^n)_\#\rho^0_{\Delta x}\bigr) 
\\[2mm]
&\ds + W_p\bigl(Z(t^n)_\#\rho^0_{\Delta x},Z(t^n)_\#\rho^{ini}\bigr).
\end{array}
\end{equation}
We will bound each term of the right hand side separately.

{\bf Initial datum}. Let us first consider the last term in the right hand side of \eqref{ineq}.
We have
$$
W_p(Z(t^n)_\#\rho^0_{\Delta x},Z(t^n)_\#\rho^{ini}) \leq L_Z(t^n,0) W_p(\rho_{\Delta x}^0,\rho^{ini}).
$$
Indeed, let $\pi$ be an optimal map in $\Gamma_0(\rho_{\Delta x}^0,\rho^{ini})$, i.e.
$$
W_p(\rho_{\Delta x}^0,\rho^{ini})^p = \int_{\RR^d\times\RR^d} |x-y|^p \pi(dx,dy).
$$
Then $\gamma=(Z(t^n),Z(t^n))_\# \pi$ is a map with marginals $Z(t^n)_\#\rho_{\Delta x}^0$ and
$Z(t^n)_\#\rho^{ini}$. It implies
$$
W_p\big(Z(t^n)_\#\rho^0_{\Delta x},Z(t^n)_\#\rho^{ini}\big)^p \leq \int_{\RR^d\times\RR^d} |x-y|^p \gamma(dx,dy) 
= \int_{\RR^d\times\RR^d} |Z(t^n,x)-Z(t^n,y)|^p \pi(dx,dy).
$$
From Lemma \ref{lem:JacobZ}, we know that the flow $Z$ is Lipschitz continuous with Lipschitz constant $L_Z(t^n,0)$. 
Thus
$$
W_p\bigl(Z(t^n)_\#\rho^0_{\Delta x},Z(t^n)_\#\rho^{ini}\bigr) \leq L_Z(t^n,0) 
\left(\int_{\RR^d\times\RR^d} |x-y|^p \pi(dx,dy)\right)^{1/p}.
$$
Precisely, using inequality \eqref{bound1Z} in Lemma \ref{lem:JacobZ}, we deduce 
\begin{equation}\label{dWinit}
W_p\bigl(Z(t^n)_\#\rho^0_{\Delta x},Z(t^n)_\#\rho^{ini}\bigr) \leq e^{\int_0^t \alpha(s)ds} W_p(\rho_{\Delta x}^0,\rho^{ini}).
\end{equation}

Now, for $\rho^{ini}\in \calP_p(\RR^d)$, we recall the definition
\begin{equation}\label{init}
\rho_{\Delta x}^0 = \sum_{J\in \ZZ^d} \rho_{J}^0 \delta_{x_J}, \quad \mbox{ with }\quad
\rho_{J}^0 = \int_{C_J} \rho^{ini}(dx) {= \rho^{ini}(C_J)}.
\end{equation}
Let us define $\tau:[0,1]\times \RR^d \to \RR^d$ by
$\tau(\sigma,x) = \sigma x_J + (1-\sigma) x$, for $x \in C_J$.
We have that $\tau(0,\cdot)=\mathrm{id}$ and $\tau(1,\cdot)_\# \rho^0 = \rho_{\Delta x}^0$. Thus 
\begin{equation}
\label{eq:wp:condition:initiale}
\begin{split}
W_p(\rho_{\Delta x}^0,\rho^{ini})^p &\leq  
\int_{\RR^d} |x-y|^p \, (\mathrm{id}\times\tau(1,\cdot))_\# \rho^{ini}(dx,dy) 
\\
& \leq  \sum_{J\in \ZZ^d} \int_{C_J} |x-x_J|^p \,\rho^{ini}(dx),
\end{split}
\end{equation}
where we use \eqref{init} for the last inequality.
We deduce $W_p(\rho_{\Delta x}^0,\rho^{ini}) \leq \Delta x$.
Injecting this latter inequality into \eqref{dWinit}, we obtain
\begin{equation}\label{boundinit}
W_p\bigl(Z(t^n)_\#\rho^0_{\Delta x},Z(t^n)_\#\rho^{ini}\bigr) \leq e^{\int_0^{t^n} \alpha(s)ds} \Delta x.
\end{equation}

{\bf Second term.} For the second term of the right hand side of \eqref{ineq}, 
by the standard property \eqref{eqWasser} of the Wasserstein distance, one has 
$$
W_p\bigl(Y^n{}_\#\rho_{\Delta x}^0,Z(t^n)_\# \rho_{\Delta x}^0\bigr) \leq 
\|Y^n-Z(t^n)\|_{L^p(\rho_{\Delta x}^0)} \leq \sup_{J\in \ZZ^d} |Y^n(x_J)-Z(t^n,x_J)|.
$$
Then applying Lemma \ref{lem:ZY}, we deduce that there exists a non-negative constant $C$ such that
\begin{equation}\label{ineq1}
\begin{split}
W_p(Y^n{}_\#\rho_{\Delta x}^0,Z(t^n)_\# \rho_{\Delta x}^0) &\leq C e^{2 \int_0^{t^n} \alpha(s)ds} \sqrt{t^n {a_{\infty}\Delta t}}
\\
&\leq C e^{2 \int_0^{t^n} \alpha(s)ds} \sqrt{t^n \Delta x},
\end{split}
\end{equation}
where 
we used again the CFL condition \eqref{CFL}
and where
$C$ only depends on $d$ and $a_{\infty}$.

{\bf First term.} 
We consider finally the first term in the right hand side of \eqref{ineq}. By
\eqref{eqWasser:omega},
$$
W_p\big(X^n{}_\#\PP_{\rho^0},\hat{Y}^n{}_\# 
\PP_{
\rho^0}\big) \leq \EE_{\rho^0} \bigl[ \vert X^n - \hat{Y}^n
\vert^p \bigr]^{1/p}.
$$
From Lemma \ref{lem:estim1}, we deduce that there exists a 
constant $C_{p}$, only depending on $p$, $d$ and $a_{\infty}$, such that
\begin{equation}\label{ineq2}
W_p\big(X^n{}_\#\PP_{\rho^0},\hat{Y}^n{}_\# 
\PP_{
\rho^0}\big)
\leq 
C_p \, e^{\int_0^{t^n} \alpha(s)\,ds}\,
\Bigl( \sqrt{t^{n} \Delta x}
+ \Delta x
\Bigr).
\end{equation}

{\bf Conclusion.} 
Injecting inequalities \eqref{boundinit}, \eqref{ineq1} and \eqref{ineq2}
into \eqref{ineq}
we deduce, for all $n\in \NN^*$,
\[
W_p\bigl(\rho^n_{\Delta x},\rho(t^n)\bigr) \leq C_{p} \, e^{{2} \int_0^{t^n} \alpha(s)ds} 
\Bigl( \sqrt{t^{n} \Delta x}
+ \Delta x
\Bigr),
\]
for a new value of $C_{p}$.

\begin{remark}
\label{rem:a:Lip:3}
When 
$a$ is $L$-Lipschitz in time (uniformly in space) and 
$a(s,\cdot)$ is replaced by $a(t^n,\cdot)$ in the 
definition 
of the upwind scheme
\eqref{def:ai}, 
Remarks
\ref{rem:a:Lip} 
and 
\ref{rem:a:Lip:2}
say that the final result holds true but with 
$\alpha$ replaced by $\alpha + L \Delta t$
and with $\int_{0}^{t_{n}} \alpha(s) ds$ replaced by 
the Riemann sum $\Delta t \sum_{k=0}^{n-1} \alpha(t^k)$. 
\end{remark}

\section{One-dimensional examples}
\label{sec:num}

In the aim to show the optimality of the present result, we perform both an exact computation 
of the error in a very simple case and provide some numerical simulations, 
in dimension 1.
We first recall that, in the one-dimensional case, the expression of the Wasserstein distance simplifies.
Indeed, any probability measure $\mu$ on the real line $\RR$
can be described thanks to its cumulative distribution function
$M(x)=\mu((-\infty,x ])$, which is a right-continuous
and non-decreasing function with $M(-\infty)=0$ and $M(+\infty)=1$.
Then we can define the generalized inverse $F_\mu$ of $M$ (or monotone rearrangement of $\mu$)
by $F_\mu(z)=M^{-1}(z):=\inf\{x\in \RR / M(x) > z\}$; 
it is a right-continuous and non-decreasing function, defined on $[0,1]$. 
For every non-negative Borel-measurable map $\xi : \RR \rightarrow \RR$, we have 
$$
\int_\RR \xi(x) \mu(dx) = \int_0^1 \xi(F_\mu(z))\,dz.
$$
In particular, $\mu\in \calP_p(\RR)$ if and only if 
$F_\mu\in L^p((0,1))$.
Moreover, in the one-dimensional setting, there exists a unique optimal transport plan
realizing the minimum in \eqref{defWp}.
More precisely, if $\mu$ and $\nu$ belong to $\calP_p(\RR)$, with monotone rearrangements
$F_\mu$ and $F_\nu$, then $\Gamma_0(\mu,\nu)=\{(F_\mu,F_\nu)_\# {\mathbb{L}}_{(0,1)}\}$ where 
${\mathbb{L}}_{(0,1)}$ is the restriction to $(0,1)$ of the Lebesgue measure.
Then we have the explicit expression of the Wasserstein distance (see \cite{rachev,Villani2}) 
\begin{equation}\label{dWF-1}
W_p(\mu,\nu) = \left(\int_0^1 |F_\mu(z)-F_\nu(z)|^p\,dz\right)^{1/p},
\end{equation}
and the map $\mu \mapsto F_\mu$ is an isometry between $\calP_p(\RR)$ and the convex subset
of (essentially) non-decreasing functions of 
$L^p((0,1))$.

We will use this expression \eqref{dWF-1} of the Wasserstein distance in dimension 1
in our numerical simulations to estimate the numerical error of the upwind scheme \eqref{dis_num}.
This scheme in dimension 1 on a Cartesian mesh reads, with time step 
$\Delta t$ and cell size $\Delta x$:
$$
\rho_j^{n+1} = \rho_j^n - \frac{\Delta t}{\Delta x}\left((a_{j}^n)^+ \rho_j^n 
-(a_{j+1}^n)^- \rho_{j+1}^n - (a_{j-1}^n)^+ \rho_{j-1}^n
+(a_{j}^n)^- \rho_j^n \right).
$$
With this scheme, we define the probability measure $\rho_{\Delta x}^n =\sum_{j\in \ZZ} \rho_j^n\delta_{x_j}$.
Then the generalized inverse of $\rho_{\Delta x}^n$, denoted by $F_{\Delta x}$, is given by
$$
F_{\Delta x}(z) = x_{j+1}, \qquad \mbox{for } z\in \Big[\sum_{k\leq j}\rho_{k}^n,\sum_{k\leq j+1}\rho_{k}^n\Big).
$$

\subsection{Proof of the optimality}

We here consider the very simple case where $a = 1$ and $\rho^{ini} = \delta_0$, thus $\rho_j^0 = \delta_{0j}$. 
In this case, the solution is given by 
$\rho(t) = \delta_t$. For the sake of simplicity, we choose $\Delta t$ and $\Delta x$ such that $\Delta t/\Delta x = 1/2$. 
The numerical scheme thus simplifies to $\rho_j^{n+1} = \rho_j^n -  1/2(\rho_j^n - \rho_{j-1}^n)$, and it is a simple exercise 
to show that then the numerical solution is 
\[
\rho_j^n = \left\{\begin{array}{l}
0 \mbox{ if } j < 0, 
\vspace{3pt}
\\
\binom{n}{j} (1/2)^n \mbox{ if } 0 \leq j \leq n, 
\vspace{3pt}
\\
0 \mbox{ if } j > n. 
\end{array}\right. 
\]
For any discrete time $t^n$, $n \in \NN$, $W_1(\rho(t), \rho_{\Delta x}^n)$ is the sum over $j$ of the distance 
$|j\Delta x - n \Delta t|$ of the cell number $j$ to the position of the Dirac mass of the exact solution, multiplied by 
the mass associated with this cell, $\rho_j^n$: 
\[
W_1(\rho(t^n), \rho_{\Delta x}^n) = \sum_{j = 0}^{n} \rho_j^n | j \Delta x - n \Delta t| 
= \sum_{j = 0}^{n} \binom{n}{j} (1/2)^n | j \Delta x - n \Delta t|
\]
The right-hand side may be written as 
$\Delta x \times \EE[\vert S_{n} - \EE(S_{n}) \vert]$, 
where $S_{n}$ is a binomial random variable
with $n$ as number of trials and $1/2$ as parameter of success. 
Recalling that the variance of $S_{n}$ is $n/4$, 
we know from the central limit theorem that 
\begin{equation*}
\lim_{n \rightarrow \infty}
\frac{2}{\sqrt{n}}
\EE[\vert S_{n} - \EE(S_{n}) \vert]
= \frac{1}{\sqrt{2 \pi}} \int_{\RR} \vert x \vert \exp \bigl( - \frac{x^2}{2} \bigr) dx
= \frac{\sqrt{2}}{\sqrt{\pi}} \int_{0}^{\infty}
x \exp \bigl( - \frac{x^2}{2} \bigr) dx
= \frac{\sqrt{2}}{\sqrt{\pi}}. 
\end{equation*}
Therefore,
\[
W_1(\rho(t^n), \rho_{\Delta x}^n) 
\sim_{n \rightarrow \infty} \frac1{\sqrt{2 \pi}} \sqrt{n \Delta x^2}
= \frac1{\sqrt{\pi}} \sqrt{t^n \Delta x},
\]
which proves the optimality of the one half order of convergence. Remark furthermore that, due to the linearity of both the 
equation and the scheme, this provides a direct proof of the convergence to the one half order of the scheme with any initial 
probability measure datum (when the velocity $a$ is constant, and at least when $a \Delta t/\Delta x = 1/2$). 
\vskip 4pt

Of course, one may bypass the use of the central limit theorem 
and perform the computations explicitly.
Choose for instance $n$ of the form $n = 2k$, $k \in \NN$. Then, thanks to the parity 
of the binomial coefficients
and to the fact that, for $j=k$, $2k \Delta t = k \Delta x = j \Delta x$, we have
\begin{align*}
&W_1(\rho(t^n), \rho_{\Delta x}^n) 
\\
&= 2 \sum_{j = 0}^{k-1} \binom{2k}{j} (1/2)^{2k} (2k \Delta t - j \Delta x ) 
= 2 \sum_{j = 0}^{k-1} \binom{2k}{j} (1/2)^{2k} (k \Delta x - j \Delta x ) \\
&= 2\Delta x \sum_{j = 0}^{k-1} \binom{2k}{j} (1/2)^{2k} (k - j) = 2\Delta x \left( k \sum_{j = 0}^{k-1} 
\binom{2k}{j} (1/2)^{2k} - k\sum_{j = 1}^{k-1} \binom{2k-1}{j-1} (1/2)^{2k-1}\right) \\
&= 2k\Delta x \left( \sum_{j = 0}^{k-1} 
\binom{2k}{j} (1/2)^{2k} - \sum_{j = 0}^{k-2} \binom{2k-1}{j} (1/2)^{2k-1} \right). 
\end{align*}
Using the two identities 
\begin{equation*}
\begin{split}
&2\sum_{j = 0}^{k-1} \binom{2k}{j} (1/2)^{2k} =  1 - \binom{2k}{k} (1/2)^{2k},
\\
&2\sum_{j = 0}^{k-2} \binom{2k-1}{j} (1/2)^{2k-1} =  1 - 2\binom{2k-1}{k-1} (1/2)^{2k-1},
\end{split}
\end{equation*} 
this rewrites 
\begin{align*}
W_1(\rho(t^n), \rho_{\Delta x}^n) &= k \Delta x \left( 1 - \binom{2k}{k} (1/2)^{2k} - 1 + \binom{2k-1}{k-1} (1/2)^{2k-2}\right) 
\\
&= k \Delta x \left( 4\binom{2k-1}{k-1} - \binom{2k}{k} \right) (1/2)^{2k} 
= k \Delta x \binom{2k}{k}  (1/2)^{2k}. 
\end{align*}
From Stirling's formula, we thus recover that 
\[
W_1(\rho(t^n), \rho_{\Delta x}^n) \sim_{n \rightarrow \infty} k \Delta x \frac{4^k}{\sqrt{k \pi}} (1/2)^{2k} = \frac{1}{\sqrt{\pi}} \sqrt{t^n \Delta x}. 
\]

\subsection{Numerical illustration}

We present in the following several numerical examples for which we compute the numerical
error in the Wasserstein distance $W_1$ using formula \eqref{dWF-1}.
For these computations, we choose the final time $T=2$ and the computational domain is $[-2.5,2.5]$.
We compute the error in the Wasserstein distance $W_1$ for different space and time step to 
estimate the convergence order.

{\bf Example 1.} 
We consider a velocity field given by $a(t,x)=1$ for $x<0$ and $a(t,x)=\frac{1}{2}$ for $x\geq 0$.
Since $a$ is non-increasing, $a$ satisfies the OSL condition \eqref{eq:OSL}.
For this example, we choose the initial datum $\rho^{ini}=\delta_{x_0}$ with $x_0=-0.5$.
Then the solution to the transport equation \eqref{eq:transp} is given by
$$
\rho(t,x) = \delta_{t+x_0}(x) \quad \mbox{ for } t<-x_0;\quad
\rho(t,x) = \delta_{\frac 12 (t+x_0)}(x) \quad \mbox{ for } t\geq -x_0.
$$
Then the generalized inverse is given for $z\in [0,1)$
by $F_\rho(t,z) = t+x_0$ if $t <-x_0$, $F_\rho(t,z)=\frac 12 (t+x_0)$ if $t\geq -x_0$.
Therefore, denoting $u_j^n=\sum_{k\leq j} \rho_k^n$, we can compute easily the error at time $t^n=n\Delta t$,
$$
e^n:=W_1(\rho(t^n),\rho_{\Delta x}^n) = \sum_{k\in\ZZ} \int_{u_{k-1}^n}^{u_k^n} |x_j - F_\rho(t^n,z)|dz.
$$
Then we define the numerical error as $e=\max_{n\leq T/\Delta t} e^n$.
We display in Figure \ref{fig:error} the numerical error with respect to the number
of nodes in logarithmic scale computed with this procedure for different time steps.
We observe that the computed numerical error is of order $1/2$.
This suggests the optimality of the result in Theorem \ref{TH}.
\begin{figure}[!ht]
\centering\includegraphics[width=8cm,height=5.5cm]{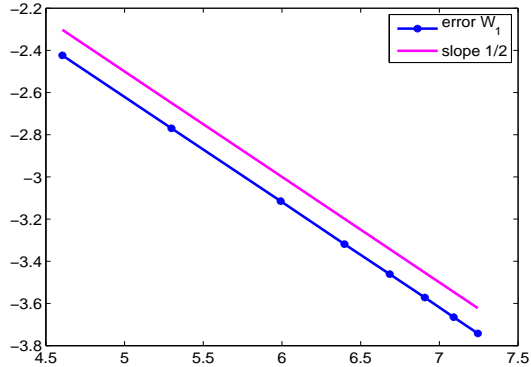}
\caption{Numerical error with respect to the number of nodes in logarithmic scale for the upwind scheme 
in Wasserstein distance $W_1$ in the case of example 1: 
initial datum given by a Dirac and velocity field $a(t,x)=1$ for $x<0$ and $a(t,x)=1/2$ for $x\geq 0$.}
\label{fig:error}
\end{figure}

{\bf Example 2.} 
We consider the same velocity field as above, given by $a(t,x)=1$ for $x<0$ and $a(t,x)=\frac{1}{2}$ for $x\geq 0$.
However, we choose for initial datum the piecewise constant function
$\rho^{ini} = \mathbf{1}_{[-1,1]}.$
Then the solution to the transport equation \eqref{eq:transp} is given by
$$
\rho(t,x) = 
\left\{\begin{array}{ll}
\mathbf{1}_{[-1+t,0)} + 2\, \mathbf{1}_{[0,t/2)} + \mathbf{1}_{[t/2,1+t/2)}, &\qquad \mbox{ for }t\leq 1, \\[2mm]
2\, \mathbf{1}_{[\frac 12 (t-1),\frac t2)} + \mathbf{1}_{[t/2,1+t/2)}, &\qquad \mbox{ for } t > 1.
\end{array}\right.
$$
We perform the numerical computation as in the first example.
Figure \ref{fig2} displays a comparison between the numerical solution $\rho_{\Delta x}$
and the exact solution $\rho$ at time $T=2$. As expected we observe numerical diffusion.
The numerical error in Wasserstein distance $W_1$ is given Figure \ref{fig2:error}-left.
We observe that the numerical error seems to be of order 1 in this case.
However, since the solution stays in $L^1$, we can estimate the numerical error in 
$L^1$, which is provided in Figure \ref{fig2:error}-right. 
We observe that this numerical error is of order 1/2.
\begin{figure}[!ht]
\centering\includegraphics[width=8.7cm,height=5.7cm]{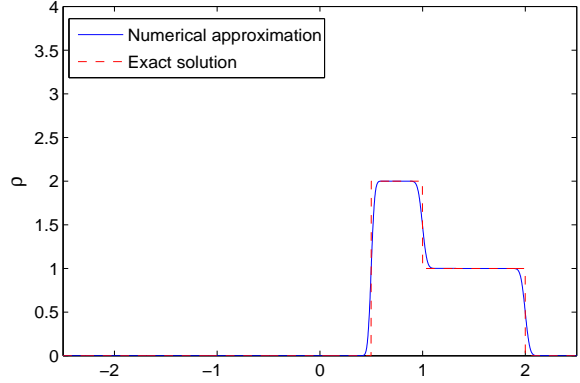}
\caption{Comparison between the numerical approximation obtained with the upwind scheme and the 
exact solution at time $T=2$ for the second example.}
\label{fig2}
\end{figure}
\begin{figure}[!ht]
\centering\includegraphics[width=8cm,height=5.5cm]{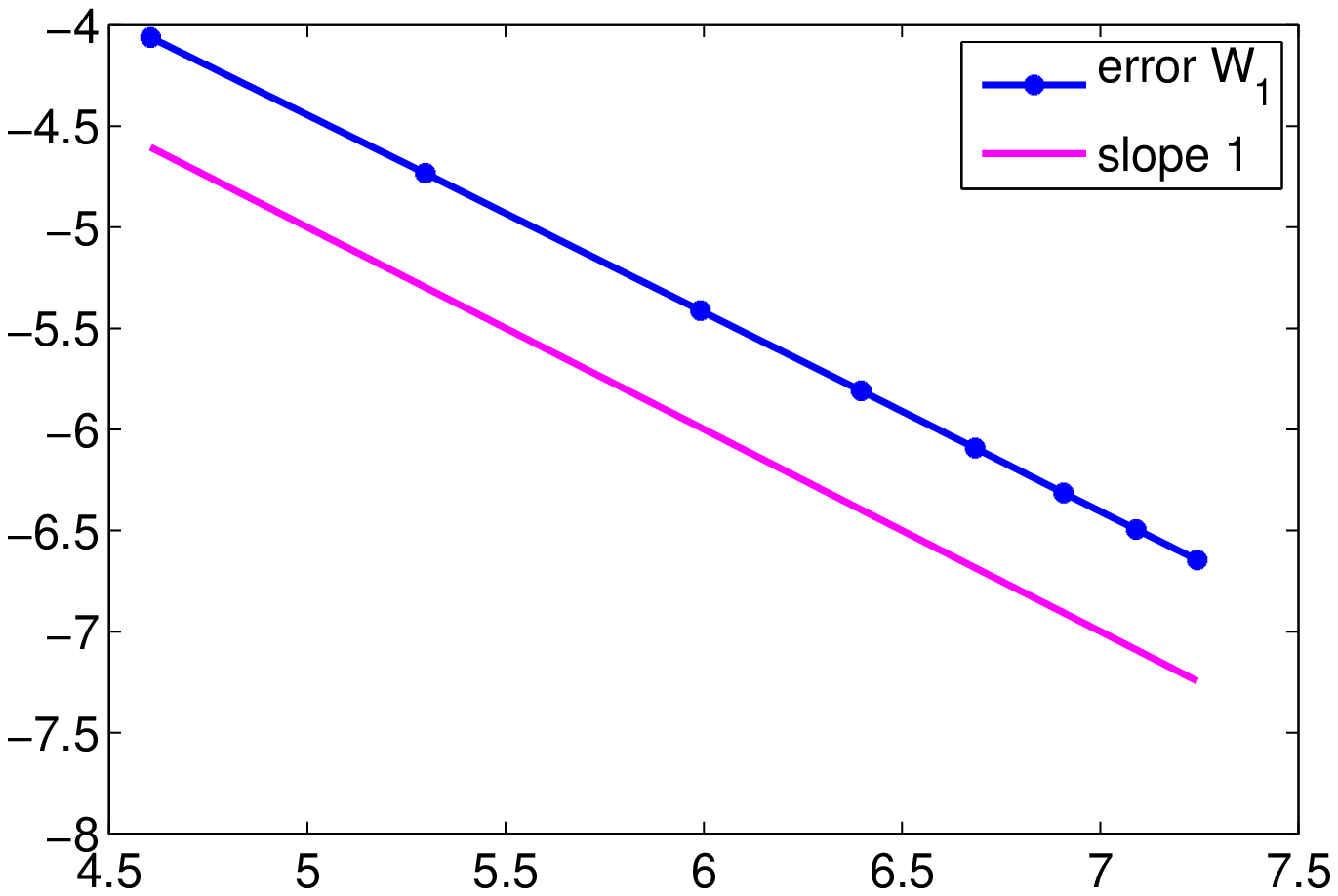}
\includegraphics[width=8cm,height=5.5cm]{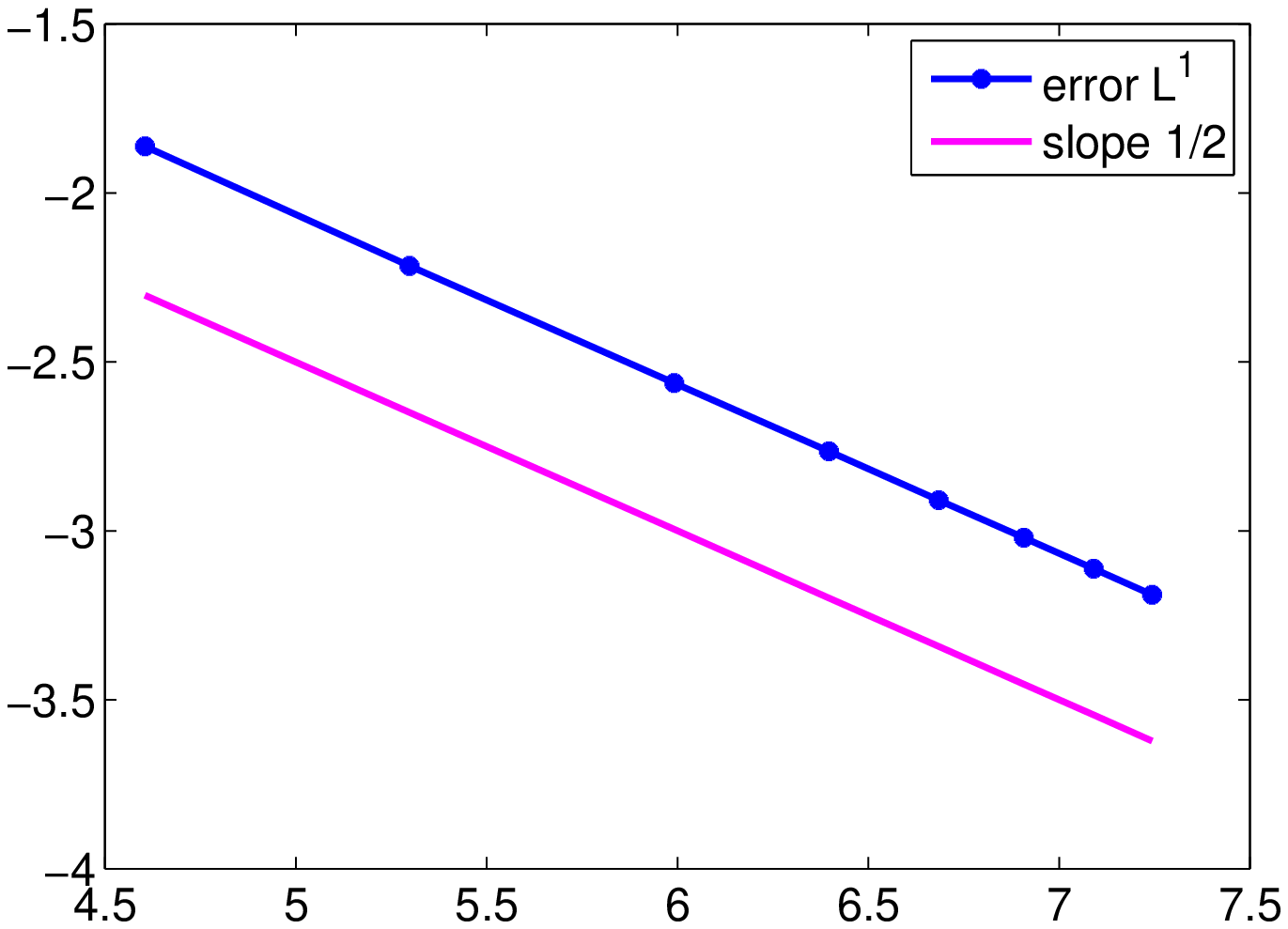}
\caption{Left: Numerical error with respect to the number of nodes
for the upwind scheme in Wasserstein distance $W_1$ 
in logarithmic scale for the second example. Right: Numerical error for the upwind scheme in $L^1$ 
norm in logarithmic scale for the second example.}
\label{fig2:error}
\end{figure}

{\bf Example 3.}
We consider the velocity field $a(t,x)=2$ for $x<\min (t,1)$ and $a(t,x)=1$ for $x\geq\min (t,1)$.
Since $a$ is non-increasing with respect to $x$, it satisfies the one-sided Lipschitz continuity condition.
The initial datum is given by: $\rho^{ini} = \mathbf{1}_{[-1,0]}$.
In this case, the solution to the transport equation \eqref{eq:transp} is given by
$$
\rho(t,x) = 
\left\{\begin{array}{ll}
\mathbf{1}_{[-1+2t,t)} + t \,\delta_{t}, &\qquad \mbox{ for }t < 1, \\[2mm]
\delta_{t}, &\qquad \mbox{ for } t \geq 1.
\end{array}\right.
$$
We deduce the expression of the generalized inverse, 
$$
F_\rho(z)= \left\{\begin{array}{ll}
(z-1+t) \mathbf{1}_{[0,1-t)} + \mathbf{1}_{[1-t,1)}, & \qquad \mbox{ for } t<1, \\[2mm]
t, &\qquad  \mbox{ for } t\geq 1.
\end{array}\right.
$$
Performing the numerical computation, we obtain the numerical error displayed in Figure \ref{fig3:error}.
We observe that in this case the order of the convergence is $1/2$.
Compared to example 2, although the initial datum is regular (piecewise constant), we have 
the formation of a Dirac delta in finite time. Then the solution is defined as a measure 
and the observed numerical order of convergence falls down to $1/2$.
\begin{figure}[!ht]
\centering\includegraphics[width=8cm,height=5.5cm]{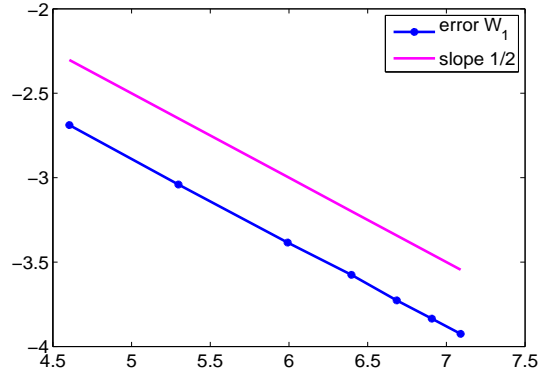}
\caption{Numerical error with respect to the number of nodes in logarithmic scale for the upwind scheme 
in Wasserstein distance $W_1$ in the case of example 3 for which a Dirac delta is created
from the initial datum $\rho^{ini} = \mathbf{1}_{[-1,0]}$.}
\label{fig3:error}
\end{figure}

As a conclusion, these numerical results seem to indicate that as long as the numerical solution 
belongs to $L^1\cap BV(\RR^d)$, the convergence of the upwind scheme is of order $1$ in 
Wasserstein distance.
However, when the velocity field is only bounded and one-sided Lipschitz continuous, the solution
might no longer be a function; for instance in example 3 above, Dirac deltas are created for $t>0$
althought the initial datum is piecewise constant.
When such singularities appear, examples 1 and 3 indicate 
that convergence order falls down to $1/2$ which shows the optimality of Theorem \ref{TH}.

\newpage

\appendix

\section{Generalization to other finite volume schemes}\label{ofvs}

For simplicity of the notations, we have presented our analysis for an upwind
scheme. In this appendix, we generalize our approach to other schemes on Cartesian grids.
With the notation above, for $J\in\ZZ^d$, we consider the scheme:
\begin{equation}\label{schemegene}
\rho_J^{n+1} = \rho_J^n - \sum_{i=1}^d \frac{\Delta t}{\Delta x_i}
\Big( g_{J+\frac 12 e_i}^{n}(\rho_J^n,\rho_{J+e_i}^n) - g_{J-\frac 12 e_i}^{n}(\rho_{J-e_i}^n,\rho_{J}^n) \Big),
\end{equation}
where we take the general form for the flux
$$
g^{n}_{J+\frac 12 e_i}(u,v) = \zeta^{n}_{J+\frac 12 e_i} u -\beta^{n}_{J+\frac 12 e_i} v.
$$
We make the following assumptions on the coefficients:
\begin{equation}\label{hypbeta}
0 \leq \zeta_{J+\frac 12 e_i}^{n} \leq \zeta_\infty, \quad 0 \leq \beta_{J+\frac 12 e_i}^{n} \leq \beta_\infty.
\end{equation}
Then, equation \eqref{schemegene} can be rewritten as
$$
\rho_J^{n+1} = \rho_J^n\biggl(1 - \sum_{i=1}^d \frac{\Delta t}{\Delta x_i} 
\bigl(\zeta_{J+\frac 12 e_i}^{n}+\beta_{J-\frac 12 e_i}^{n}\bigr) \biggr) +
\sum_{i=1}^d \frac{\Delta t}{\Delta x_i} \Big(\beta^{n}_{J+\frac 12 e_i} \rho_{J+e_i}^n 
+ \zeta_{J-\frac 12 e_i}^{n} \rho_{J-e_i}^n \Big).
$$
Assuming that the following CFL condition holds
\begin{equation}\label{CFLgene}
(\beta_\infty + \zeta_\infty) \sum_{i=1}^d \frac{\Delta t}{\Delta x_i} \leq 1,
\end{equation}
the scheme is clearly non-negative.
We define then the random characteristics as in Section \ref{sec:proba} by \eqref{defXj} 
where the transition matrix at time $n$ is now given by
$$
P_{J,L}^{n} = \left\{
\begin{array}{ll}
\ds 1 - \sum_{i=1}^d \frac{\Delta t}{\Delta x_i} 
(\zeta^{n}_{J+\frac 12 e_i}+\beta^{n}_{J-\frac 12 e_i}) \qquad & \mbox{when } L=J, 
\vspace{5pt}
\\
\ds \frac{\Delta t}{\Delta x_i} \zeta^{n}_{J+\frac 12 e_i}
\qquad & \mbox{when } L=J+e_i, \quad \mbox{for } i=1,\ldots,d, 
\vspace{5pt}
\\
\ds \frac{\Delta t}{\Delta x_i} \beta^{n}_{J-\frac 12 e_i}
\qquad & \mbox{when } L=J-e_i, \quad \mbox{for } i=1,\ldots,d, 
\vspace{5pt}
\\
0 & \mbox{otherwise.}
\end{array}
\right.
$$
It is clear that Lemma \ref{lem:transport:rho} and Proposition \ref{propZ} (ii) hold true 
with this random characteristics.
We compute
$$
\E^n_J\bigl(X^{n+1}-X^n\bigr) = \sum_{i=1}^d \left( \Delta x_i \, \zeta^{n}_{J+\frac 12 e_i} \frac{\Delta t}{\Delta x_i}
- \Delta x_i \, \beta^{n}_{J-\frac 12 e_i} \frac{\Delta t}{\Delta x_i}\right)e_i.
$$
Thus
$$
\E^n_J\bigl(X^{n+1}-X^n\bigr) = \Delta t \sum_{i=1}^d \bigl(\zeta^{n}_{J+\frac 12 e_i} - \beta^{n}_{J-\frac 12 e_i}\bigr)e_i.
$$
We deduce the following result:
\begin{proposition}
Under the assumptions of Theorem \ref{TH} on $\rho^{ini}$ and $a$, 
assume further that the bounds \eqref{hypbeta} and the CFL condition \eqref{CFLgene} hold true.

If moreover the weights $(((\zeta^n_{J+e_{i}/2})_{i=1,\dots,d})_{J \in \ZZ^d})_{n \in \NN}$ and $(((\beta^n_{J+e_{i}/2})_{i=1,\dots,d})_{J \in \ZZ^d})_{n \in \NN}$ satisfy
\begin{equation}\label{relbeta}
\zeta^n_{J+\frac 12 e_i} - \beta^n_{J-\frac 12 e_i} = {a_i}^n_{J},
\end{equation}
where ${a_i}^n_J$ is given in \eqref{def:ai},
then, the result of Theorem \ref{TH} still holds true for the scheme \eqref{schemegene}.
\end{proposition}
Indeed, thanks to \eqref{relbeta}, Proposition \ref{propZ} (i) holds true.
Thus, we can redo the proof of Theorem \ref{TH} in this framework.

\begin{example} { \ }
\begin{itemize}
\item We first observe that, for $\zeta^{n}_{J+\frac 12 e_i} = ({a_i}^n_J)^+$ and 
$\beta^{n}_{J-\frac 12 e_i} = -({a_i}^n_J)^-$,
\eqref{relbeta} is satisfied. This choice corresponds to the upwind scheme \eqref{dis_num}
considered in this paper.
\item If we now consider the Rusanov scheme, we then have
$\zeta_{J+\frac 12 e_i}^{n} = \frac 12 ({a_i}^n_J + a_\infty)$ and 
$\beta_{J-\frac 12 e_i}^{n} = \frac 12 (-{a_i}^n_J+ a_\infty)$.
We easily check that \eqref{hypbeta}
and \eqref{relbeta} are satisfied.
Thus our result also shows that the Rusanov scheme, when applied to a conservative transport
equation with a velocity field that is only $L^\infty$ and OSL, has an order $1/2$
in distance $W_p$, $p\geq 1$.
\end{itemize}
\end{example}

Finally, as pointed out in Remark \ref{otherupwind}, we observe that for another 
traditional upwind scheme given by:
$$
\rho_{J}^{n+1} = \rho_{J}^n - \sum_{i=1}^{d} \frac{\Delta t}{\Delta x_i}
\Big(({a_i}^n_{J+\frac 12 e_i})^+ \rho_{J}^n - ({a_i}^n_{J+\frac 12 e_i})^- \rho_{J+e_i}^n 
-({a_i}^n_{J-\frac 12 e_i})^+ \rho_{J-e_i}^n + ({a_i}^n_{J-\frac 12 e_i})^- \rho_{J}^n \Big),
$$
where ${a_i}^n_{J+\frac 12 e_i} = \frac{1}{\Delta t}\int_{t^n}^{t^{n+1}} a_i(s,x_{J+\frac 12 e_i})\,ds$,
we have
$$
\E^n_J(X^{n+1}-X^n) = \int_{t^n}^{t^{n+1}} \sum_{i=1}^d \biggl( a_i(s,X^n+\frac{\Delta x_i}{2} e_i)^+
- a_i(s,X^n-\frac{\Delta x_i}{2} e_i)^-\biggr)e_i\, ds.
$$
Then the statements of Proposition \ref{propZ} (i) does not hold. 
Consequently, we cannot use the techniques developed in this paper.

\section{Application of the technique to a scheme on unstructured meshes}\label{sec:unstr}

In this section, we explain shortly how to obtain the error estimate for a forward semi-Lagrangian 
scheme defined on an unstructured mesh. 
For the sake of simplicity, we present the case of a triangular mesh in dimension 2, 
but this approach can be easily extended to any mesh made of simplices, in any dimension. 
Basic references on forward semi-Lagrangian schemes are 
\cite{CRS} and \cite{Denavit} (although they concern schemes on structured quadrilateral meshes). 

\subsection{Numerical algorithm}

Let us consider a triangular mesh $\calT = (T_k)_{k\in \ZZ}$ with nodes $(x_i)_{i\in \ZZ}$. We assume this mesh to be conformal: A summit cannot belong to an open edge of the grid. 
The triangles $(T_k)_{k \in \ZZ}$ are assumed to satisfy $\bigcup_{k\in\ZZ} T_k = \RR^2$ and $T_k \cap T_l = \emptyset$ if $k \neq l$ (in particular, the cells are here not assumed to be closed nor open). For any triangle $T$ with summits $x$, $y$, $z$, we will use also the notation
$(x,y,z) = T$. We denote by $\calV(T) = \calV(x,y,z)$ the area of this
triangle, and $h(T)$ its height (defined as the minimum of the three heights of the triangle $T$).
We make the assumption that the mesh satisfies $\hbar:=\inf_{k\in \ZZ} h(T_k) >0$.

For any node $x_i$, $i \in \ZZ$, we denote by $K(i)$ the set of 
indices indexing triangles that have $x_i$ as a summit, and we denote by $\calT_i$ the set of all 
triangles of $\calT$ that have $x_i$ as a summit: thus $\calT_i = \{T_k ; k \in K(i) \}$. 

For any triangle $T_k$, $k \in \ZZ$, we denote by $I(k) = \{I_1(k), I_2(k), I_3(k)\}$ 
the set of indices indexing the summits of $T_k$
(for some arbitrary order, whose choice has no importance for the sequel). 
\vskip 4pt

Here is the derivation of the forward semi-Lagrangian scheme, 
whose rigorous definition is given next, in 
\eqref{schemeT}. Let us emphasize that this is not a finite volume scheme. 

\begin{itemize}
\item 
For an initial distribution $\rho^{ini}$
of the PDE \eqref{eq:transp}, define
the probability weights $(\rho^0_{i})_{i \in \ZZ}$
through the following procedure: 
Consider the one-to-one mapping $\kappa : 
\ZZ \ni i \mapsto \kappa(i) \in \ZZ$ such that, for 
each $i \in \ZZ$,  
the node $x_{i}$ belongs to 
the triangle $T_{\kappa(i)}$
(such a triangle is unique since 
the cells  $(T_k)_{k\in \ZZ}$
are disjoint); 
$\kappa$ is thus a canonical way to associate a cell 
with a node; then,
for all $i \in \ZZ$, 
let
$\rho^0_{i} = \rho^{ini}(T_{\kappa(i)})$.
Observe from 
\eqref{eq:wp:condition:initiale}
 that 
$\rho^0_{\Delta x} = \sum_{j \in \ZZ} \rho^0_{j} \delta_{x_{j}}$ is an approximation of $\rho^{ini}$. 
\item Assume that, for a given $n \in \NN$,  
we already have 
probability weights  
$(\rho_i^n)_{i\in\ZZ}$ 
such that 
$\rho^n_{\Delta x} =
\sum_{j \in \ZZ} \rho^n_{j} \delta_{x_{j}}$
is an approximation of $\rho(t^n,\cdot)$, 
where $\rho$ is the solution to 
\eqref{eq:transp}
with 
$\rho^{ini}$
as initial condition.
Similar to \eqref{def:ai}, let us denote 
$a_i^n= {\Delta t}^{-1} \int_{t^n}^{t^{n+1}} a(s,x_i)\,ds$, and $x_i^n= x_i + a_i^n \Delta t$, 
for $i \in \ZZ$.
Under the CFL-like condition
\begin{equation}\label{CFLT}
a_\infty \Delta t \leq \hbar,
\end{equation}
$x_i^n$ belongs to one (and only one) 
of the elements of $\calT_i$.
We denote by $k_i^n$ the index of this triangle: $x_i^n\in T_{k_i^n}$.
\item The basic idea now is to use a linear splitting rule between the summits of 
the triangle $T_{k_i^n}$: the mass $\rho_i^n$ is sent to these three points 
$x_{I_1(k_i^n)}$, $x_{I_2(k_i^n)}$, $x_{I_3(k_i^n)}$ according to the 
{\em barycentric coordinates} of $x_i^n$ in the triangle. 
In some sense, this scheme is a natural extension of the one-dimensional upwind scheme to greater dimensions (see the interpretation of the one-dimensional upwind scheme provided in Remark \ref{ex1D}). 
\end{itemize}
\begin{center}
\begin{tikzpicture}
\draw (0,0) -- (5,0) -- (3,4) -- (0,0);
\draw (0,0) node[below]{$x_i = x_{I_1(k_i^n)}$};
\draw (5,0) node[below]{$x_{I_2(k_i^n)}$};
\draw (3,4) node[above]{$x_{I_3(k_i^n)}$};
\draw (2.5,2) node[below]{$x_i^n$};
\draw [->] (0,0) -- (2.5,2);
\draw [dashed] (5,0) -- (2.5,2) -- (3,4);
\end{tikzpicture}
\end{center}
Let $T = (x,y,z) \in \calT$, and $\xi \in T$. 
We define the barycentric coordinates of $y$ with respect 
to $x$, $y$ and $z$, $\lambda_{x}^{T}$, 
$\lambda_{y}^{T}$ and $\lambda_{z}^{T}$: 
\begin{equation}\label{lambdaetal}
\lambda_{x}^{T}(\xi) = \frac{\calV(\xi,y,z)}{\calV(T)}, \quad 
\lambda_{y}^{T}(\xi) = \frac{\calV(\xi,x,z)}{\calV(T)}, \quad 
\lambda_{z}^{T}(\xi) = \frac{\calV(\xi,x,y)}{\calV(T)}, \quad 
\end{equation}
and then have $\xi = \lambda_{x}^{T}(\xi) x + \lambda_{y}^{T}(\xi) y 
+ \lambda_{z}^{T}(\xi) z$. 
Note also that $\lambda_{x}^{T}(\xi) + \lambda_{y}^{T}(\xi) 
+ \lambda_{z}^{T}(\xi) = 1$. Therefore, we have the following fundamental 
property, which will be used in the sequel: 
\begin{equation}
\label{fundapropo}
\lambda_{x}^{T}(\xi) (x - \zeta) + 
\lambda_{y}^{T}(\xi) (y - \zeta)
+ \lambda_{z}^{T}(\xi) (z - \zeta) = \xi - \zeta,
\end{equation}
for any $\zeta \in \RR^2$. \\

Considering $x_i^n \in T_{k_i^n}$, we will use the barycentric coordinates of $x_i^n$ with respect to the 
summits
$(x_j)_{j \in I(k_i^n)}$ of $T_{k_i^n}$. For 
notational convenience, let us denote 
\[
\lambda_{i,j}^n = \lambda_{x_j}^{T} (x_i^n) \qquad 
\textrm{\rm when} \quad T = T_{k_i^n}. 
\]
The numerical scheme reads:
\begin{equation}\label{schemeT}
\rho_j^{n+1} = \sum_{i\in \gamma(j)} \rho_i^n \lambda_{i,j}^n, 
\qquad j \in \ZZ, \ n \in \NN,
\end{equation}
where, for a given $j \in \ZZ$, we denote by $\gamma(j)$ the set of all indices $i \in \ZZ$ indexing
nodes $x_i$ such that $x_i + a_i^n \Delta t$ belongs to a triangle that has $x_j$ as a summit : 
\[
\gamma(j) = \{ i \in \ZZ \,  / \mbox{ there exists } k \in K(j) \mbox{ such that } x_i + a_i^n \Delta t \in T_k\}. 
\]

\subsection{Probabilistic interpretation}

As in Section \ref{sec:proba}, we define a random characteristic associated 
to the scheme \eqref{schemeT}.
Letting $\Omega= \ZZ^\NN$ and 
defining
the canonical process $(I^n)_{n\in\NN}$
as 
we defined 
$(K^n)_{n\in\NN}$
above (the definition is the same but we prefer to use 
the letter $I$ instead of $K$; we make this clear right below), we equip $\Omega$ with the Kolmogorov 
$\sigma$-field $\calA$ and with a collection of probability measures
$(\PP_\mu)_{\mu\in\calP(\ZZ)}$, such that, for each $\mu\in \calP(\ZZ)$, 
$(I^n)_{n\in\NN}$ is a 
time-inhomogeneous Markov chain under $\PP_\mu$, with 
$\mu$ as initial distribution and with 
transition matrix:
\begin{equation}
\label{defPij}
P_{i,j}^n =
\left\{\begin{array}{ll}
\ds \lambda_{i,j}^n &\textrm{when } j \in I(k_i^n), \\
0 \quad &\textrm{otherwise},
\end{array}\right.
\end{equation}
that is to say, more precisely, 
\[
P_{i,j}^n =
\left\{\begin{array}{ll}
\ds \lambda_{x_{j}}^{T_{k_i^n}}(x_i^n) &\textrm{when} \ j \in  I(k_i^n), 
\\
0 \quad &\textrm{otherwise},
\end{array}\right.
\]
with the notation in \eqref{lambdaetal}.
Pay attention that the chain $(I^{n})_{n \in \NN}$ here takes values in the set of indices 
indexing the nodes of the grid whilst the chain 
$(K^{n})_{n \in \NN}$
used in the analysis 
of the upwind scheme (see Section 
\ref{sec:proba}) takes values in the set of indices indexing the 
cells of the grid. This is the rationale for using different letters.

Then, we let the random characteristics be the sequence of random variables 
$(X^n)_{n\in\NN}$ from $(\Omega,\calA)$ into $\RR^2$ defined by
\begin{equation}\label{defcharac}
\forall\, n\in \NN, \ \forall \omega\in \Omega,\quad X^n(\omega) = x_{I^n(\omega)}.
\end{equation}

We now check that Proposition \ref{propZ} 
still holds true with this definition of the random characteristics:

\begin{proposition}\label{propZ2}
Let $(X^n)_{n \in \NN}$ be the random characteristics defined by
\eqref{defPij}--\eqref{defcharac}.

(i) 
Defining $\rho_{\Delta x}^n = \sum_{j\in \ZZ} \rho_j^n \delta_{x_j}$, we have
 $\rho^n_{\Delta x} = X^n\,_\#\PP_{\rho_{\Delta x}^0}$.

(ii) For all $j\in \ZZ$, we have, with probability one under $\PP_j$, 
$\E^n_j(X^{n+1}-X^n) = {a^n_{I^n}} \Delta t$.

\end{proposition}
\begin{proof}
(i) This result follows straightforwardly from the proof of
(ii) in 
 Proposition \ref{propZ}, but  
with the transition matrix defined in \eqref{defPij}.

(ii) From a direct computation, we have
$$
\E^n_j(X^{n+1} - X^n) = \sum_{\ell \in I(k_{I^n}^n)} \lambda_{I^n,\ell}^n (x_{\ell} - x_{I^n}).
$$
Thanks to Property \eqref{fundapropo}, 
$$
\E^n_j(X^{n+1} - X^n) = x^n_{I^n}-x_{I^n} = a_{I^n}^n \Delta t,
$$
which completes the proof.
\end{proof}

\subsection{Convergence order}

By the same token as in Section \ref{sec:ordre}, we can use 
Proposition \ref{propZ2} and Lemmas \ref{lem:ZY}
and \ref{lem:estim1} to prove that the numerical scheme \eqref{schemeT} is of order $1/2$: 

\begin{theorem}\label{TH2}
Let $\rho^{ini} \in \calP_p(\RR^d)$ for $p\geq 1$. Let us assume that 
$a\in L^\infty([0,\infty);L^\infty(\RR^2))^2$
and satisfies the OSL condition \eqref{eq:OSL}.
Let $\rho=Z_\# \rho^{ini}$ be the unique measure solution to the aggregation equation 
with initial datum $\rho^{ini}$ in the sense of Theorem \ref{thPR}.
Let us consider a triangular conformal mesh $(T_k)_{k\in \ZZ}$ with nodes $(x_j)_{j\in \ZZ}$ 
such that $\hbar = \inf_{k\in \ZZ} h(T_k) >0$. We denote by 
$\Delta x$ the longest edge in the mesh. 
We define
$$
\rho_{\Delta x}^n = \sum_{j\in \ZZ} \rho_j^n \delta_{x_j},
$$
where the approximation sequence $(\rho_j^n)$ is computed thanks to the scheme \eqref{schemeT}.
We assume that the CFL condition \eqref{CFLT} holds.
Then, there exists a non-negative constant $C$, such that for all $n\in \NN^*$, 
$$
W_p(\rho(t^n),\rho_{\Delta x}^n) \leq C e^{{2}\int_0^{t^n}\alpha(s)ds}\bigl( \sqrt{t^n \Delta x} +
{\Delta x} \bigr).
$$
\end{theorem}

\begin{proof}
In order to repeat the arguments developed 
in Section \ref{sec:ordre}, the main point is to check 
the analogue of
the second inequality in  
\eqref{eq:hn}.

To do so, we need to bound $\lambda^n_{i,j}$ for 
$j \not = i$. Clearly,
\begin{equation*}
\sum_{j \not = i} \lambda^n_{i,j}
=
\frac{{\mathcal V}(x_{i}^n,x_{i},z)+
{\mathcal V}(x_{i}^n,x_{i},y)
}{{\mathcal V}(T_{k_{i}^n})},
\end{equation*}
where $y$ and $z$ are the two 
summits of $T_{k_{i}^n}$ that are different from $x_{i}$, namely 
$\{y,z\} = I(k_{i}^n) \setminus \{x_{i}\}$. 
Since $\vert x_{i}^n - x_{i} \vert \leq a_{\infty} \Delta t$,
we have
\begin{equation*}
{\mathcal V}(x_{i}^n,x_{i},y)
\leq \frac{a_{\infty} \Delta t  \Delta x}{2}, 
\quad
{\mathcal V}(x_{i}^n,x_{i},z)
\leq \frac{a_{\infty} \Delta t  \Delta x}{2},
\end{equation*}
from which we get 
\begin{equation*}
\sum_{j \not = i} \lambda^n_{i,j}
\leq \frac{a_{\infty} \Delta t  \Delta x}{{\mathcal V}(T_{k_{i}^n})} \leq 2 a_{\infty} \frac{\Delta t}{\hbar}. 
\end{equation*}
This permits to implement the strategy proposed in 
Lemma \ref{lem:estim1}, as long as the constants therein are allowed to depend upon $1/\hbar$, which shows the need for requiring $\hbar >0$.
\end{proof}

\section{An interpolation result}

\begin{proposition}[\cite{Filippo}]
There exists a constant $C$ such that, for any $f$, $g$, non-negative functions 
in $BV(\RR^d)$ such that $\int_{\RR^d} f  = \int_{\RR^d} g = 1$, it holds that 
\[
|| f - g || \leq C | f - g |_{BV}^{1/2} W_1(f,g)^{1/2},
\]
where $| \cdot |_{BV}$ denotes the $BV$ semi-norm. 
\end{proposition}

\begin{proof}
Let $h \in L^\infty(\RR^d)$ such that $|| h ||_{L^\infty} = 1$. Let $\rho$ be a smoothing kernel and 
$\rho_\varepsilon(x) = \rho(x/\varepsilon)/
\varepsilon^d$, for any $\varepsilon > 0$. 
Let us denotes $h_\varepsilon = h \star \rho_\varepsilon$, 
$f_\varepsilon = f \star \rho_\varepsilon$, $g_\varepsilon = g \star \rho_\varepsilon$, where $\star$ 
stands for the convolution product.
One has 
\[
\int_{\RR^d} h(f - g) = \int_{\RR^d} h \bigl(f - g - (f_\varepsilon - g_\varepsilon)\bigr) + 
\int_{\RR^d} h (f_\varepsilon - g_\varepsilon). 
\]
Let us now estimate both integrals in the right hand term above. On the one hand, we have
\[
\int_{\RR^d} h \bigl(f - g - (f_\varepsilon - g_\varepsilon)\bigr) \leq || h ||_{L^\infty} || f - g - (f_\varepsilon - 
g_\varepsilon) ||_{L^1} 
\leq C \varepsilon | f - g |_{BV},
\]
for $C = \int_{\RR^d} |x| \rho(x) \, dx$; indeed, 
\begin{equation*}
\begin{split}
\int_{\RR^d} |f - f_\varepsilon| &= 
\int_{\RR^d} \left| \int_{\RR^d}  \frac{(f(x) - f(x-y))\rho(y/\varepsilon)}{\varepsilon^d}\, dy \right| dx
\\ 
&\leq \int_{\RR^d} \int_{\RR^d} | f(x) - f(x - \varepsilon y)| 
\rho(y) \, dy \, dx 
\\
&= \int_{\RR^d} \biggl( \int_{\RR^d} | f(x) - f(x - \varepsilon y)| \, dx \biggr) \, \rho(y) \, dy 
\leq \varepsilon | f |_{BV} \int_{\RR^d} |y| \rho(y) \, dy , 
\end{split}
\end{equation*}
the last inequality being due to the fact that $\int_{\RR^d} | f(x) - f(x - \varepsilon y)| \, dx \leq \varepsilon  | f |_{BV} |y|$ for any $y$: see for example Remark 3.25 in \cite{AFP}. On the other hand, 
\[
\int_{\RR^d} h (f_\varepsilon - g_\varepsilon) = \int_{\RR^d} h_\varepsilon (f - g) \leq W_1(f,g) || \nabla h_\varepsilon ||_{L^\infty},
\]
where we used the identity $W_1(f,g) = \sup_{h \in \mathcal{C}^1 / || \nabla h ||_{L^\infty} \leq 1} \int_{\RR^d} (f - g) h$ (see \cite{Villani1}). 
Furthermore, 
\[ 
|| \nabla h_\varepsilon||_{L^\infty} = || h \star \nabla \rho_\varepsilon ||_{L^\infty} 
\leq || h ||_{L^\infty} || \nabla \rho_\varepsilon ||_{L^1} \leq \frac{1}{\varepsilon} || \nabla \rho ||_{L^1}. 
\]
In the end, taking $C = \max(\int_{\RR^d} |x| \rho(x) dx, \int_{\RR^d} | \nabla \rho(x)| dx)$, one gets 
\[
\int_{\RR^d} h(f - g) \leq C \biggl( \varepsilon | f - g |_{BV} + \frac{1}{\varepsilon} W_1(f,g) \biggr). 
\]
``Optimizing'' in $\varepsilon$, that is to say, taking 
$\varepsilon = W_1(f,g)^{1/2}/|f - g|_{BV}^{1/2}$ (without any loss of generality, we can assume $|f - g|_{BV}^{1/2} \not =0$), one gets 
\[
|| f - g ||_{L^1} = \sup_{h \in L^\infty, || h ||_{L^\infty} = 1}\int_{\RR^d} h(f - g) \leq C | f - g |_{BV}^{1/2} W_1(f,g)^{1/2},
\]
which completes the proof. 
\end{proof}

\bigskip
{\bf Acknowledgements.}
NV acknowledges partial support from the french ``ANR blanche'' project Kibord~: ANR-13-BS01-0004.

\end{document}